\documentclass{amsart}
\usepackage{amsmath,amsfonts,amsthm,amssymb}
\usepackage[all]{xy}
\usepackage{color}
\usepackage[final]{optional} % comm/final for comm/no comm
\usepackage{hyperref}
\newcommand\comm[1]{\opt{comm}
{{\color{red}$\blacktriangleright$\ \small\sf#1$\blacktriangleleft$}}}
\newcommand\BC{{\mathbb C}}
\newcommand\BF{{\mathbb F}}
\newcommand\BQ{{\mathbb Q}}
\newcommand\cC{{\mathcal C}}
\newcommand\cF{{\mathcal F}}
\newcommand\cL{{\mathcal L}}
\newcommand\cU{{\mathcal U}}
\newcommand\cS{{\mathcal S}}
\newcommand\cT{{\mathcal T}}
\newcommand\bB{{\mathbf B}}
\newcommand\bG{{\mathbf G}}

\newcommand\bL{{\mathbf L}}
\newcommand\bM{{\mathbf M}}
\newcommand\bP{{\mathbf P}}
\newcommand\bQ{{\mathbf Q}}
\newcommand\bS{{\mathbf S}}
\newcommand\bT{{\mathbf T}}
\newcommand\bU{{\mathbf U}}
\newcommand\bV{{\mathbf V}}
\newcommand\bX{{\mathbf X}}
\newcommand\bY{{\mathbf Y}}
\newcommand\comp[1]{^{[#1]}}
\newcommand\compF[1]{^{[#1]F}}
\newcommand\Bo{{\bB^0}}
\newcommand\Go{{\bG^0}}
\newcommand\GoF{{(\Go)^F}}
\newcommand\Gun{{\bG^1}}
\newcommand\GunF{{(\bG^1)^F}}
\newcommand\GsoF{{C_\bG(s)^{0F}}}
\newcommand\Gunu{{C_\bG(s)^0\cdot u}}
\newcommand\Gunsu{{C_\bG(s)^0\cdot su}}
\newcommand\Gso{{\bG^{\sigma 0}}}
\newcommand\Gxo{{\bG^{x 0}}}
\newcommand\Gto{{\bG^{t0}}}
\newcommand\GtoF{{(\Gto)^F}}
\newcommand\GloF{{(\bG^{l0})^F}}
\newcommand\Glo{{\bG^{l0}}}
\newcommand\Lo{{\bL^0}}
\newcommand\Lun{{\bL^1}}
\newcommand\LunF{{(\bL^1)^F}}
\newcommand\Mun{{\bM^1}}
\newcommand\Lto{{\bL^{t0}}}
\newcommand\Llo{{\bL^{l0}}}
\newcommand\LoF{{(\Lo)^F}}
\newcommand\Lso{{\bL^{\sigma 0}}}
\newcommand\LloF{{(\bL^{l0})^F}}
\newcommand\Mso{{\bM^{\sigma 0}}}
\newcommand\Po{{\bP^0}}
\newcommand\Tun{{\bT^1}}
\newcommand\TunF{{(\bT^1)^F}}
\newcommand\cTun{{\cT_1}}
\newcommand\cTunF{{\cT_1^F}}
\newcommand\Tpun{{\bT^{\prime1}}}
\newcommand\To{{\bT^0}}
\newcommand\ToF{{(\To)^F}}
\newcommand\Tpo{{\bT^{\prime 0}}}
\newcommand\TpoF{{(\bT^{\prime 0})^F}}
\newcommand\Tso{{(\bT^s)^0}}
\newcommand\Gunqss{{\cC(\Gun)_{\text{\rm qss}}}}
\newcommand\Tunrat{{\bT^{1\text{\rm rat}}}}
\newcommand\Tsigo{{\bT^{\sigma0}}}
\newcommand\Ls{{\cL_\sigma}}
\newcommand\Fq{{\BF_q}}
\newcommand\Fqbar{{\overline\BF_q}}
\newcommand\hTunsu{{\lexp h\bT\cap\Gunsu}}
\newcommand\Tunsu{{\Tun\cap\Gunsu}}
\newcommand\Lunsu{{C_\bL(s)^0\cdot su}}
\newcommand\Lunu{{C_\bL(s)^0\cdot u}}
\newcommand\qss{quasi-semi-simple}
\newcommand\eps{{\varepsilon}}
\newcommand\sF{{\sigma F}}
\newcommand\st{{\sigma t}}
\newcommand\sll{{\sigma l}}
\newcommand\lps{{l'\sigma}}
\newcommand\GlpsosF{{(\bG^\lps)^0}^F}
\newcommand\GlsosF{{(\bG^\sll)^0}^\sF}
\newcommand\NFsF{n_{F/\sF}}
\newcommand\RTG{{R_\Tun^\Gun}}
\newcommand\RLG{{R_\Lun^\Gun}}
\newcommand\ShFsF{\Sh_{F/\sF}}
\newcommand\sRLG{{\lexp *R_\Lun^\Gun}}
\newcommand\YUF{Y_{\bU,F}}
\newcommand\YUsF{Y_{\bU,\sF}}
\newcommand\sFti{{\sigma\lexp F\lambda\inv}}
\newcommand\unip{{\text{unip}}}
\DeclareMathOperator\D{\mathrm{D}}
\DeclareMathOperator\Id{\mathrm{Id}}
\DeclareMathOperator\ad{\mathrm{ad}}
\DeclareMathOperator\diag{\mathrm{diag}}
\DeclareMathOperator\Ind{\mathrm{Ind}}
\DeclareMathOperator\Irr{\mathrm{Irr}}
\DeclareMathOperator\St{\mathrm{St}}
\DeclareMathOperator\Res{\mathrm{Res}}
\DeclareMathOperator\Trace{\mathrm{Trace}}
\DeclareMathOperator\Sh{\mathrm{sh}}
\DeclareMathOperator\SL{\mathrm{SL}}
\DeclareMathOperator\GL{\mathrm{GL}}
\DeclareMathOperator\PGL{\mathrm{PGL}}
\DeclareMathOperator\SO{\mathrm{SO}}
\DeclareMathOperator\Sp{\mathrm{Sp}}
\DeclareMathOperator\Spin{\mathrm{Spin}}
\DeclareMathOperator\rad{\mathrm{rad}}
\newcommand\Qlbar{{\overline\BQ_\ell}}
\def\scal#1#2#3{\langle#1,#2\rangle_{#3}}
\newcommand\genby[1]{\mathopen<#1\mathclose>}
\newcommand\inv{^{-1}}
\newcommand\lexp[2]{\kern\scriptspace\vphantom{#2}^{#1}\kern-\scriptspace#2}
\newcommand\addots{%
        \mathinner{\mkern1mu\raise1pt\vbox{\kern7pt\hbox{.}}\mkern2mu
    \raise4pt\hbox{.}\mkern2mu\raise7pt\hbox{.}\mkern1mu}}
\title{Complements on disconnected reductive groups}
\author{F.~Digne and J.~Michel}
\address[J.~Michel]{Institut de Math\'ematiques de Jussieu -- Paris rive
gauche, Universit\'e Denis Diderot, B\^atiment Sophie Germain,
75013, Paris France.}
\email{jean.michel@imj-prg.fr}
\urladdr{webusers.imj-prg.fr/$\sim$jean.michel}
\address[F.~Digne]{Laboratoire Ami\'enois de Math\'ematique Fondamentale et 
Appliqu\'ee, CNRS UMR 7352, Universit\'e de Picardie-Jules Verne,
80039 Amiens Cedex France.}
\email{digne@u-picardie.fr}
\urladdr{www.lamfa.u-picardie.fr/digne}

\newtheorem{proposition}[equation]{Proposition}
\newtheorem{definition}[equation]{Definition}
\newtheorem{theorem}[equation]{Theorem}
\newtheorem{corollary}[equation]{Corollary}
\newtheorem{lemma}[equation]{Lemma}
\theoremstyle{definition}
\newtheorem{example}[equation]{Example}
\theoremstyle{remark}
\newtheorem{remark}[equation]{Remark}
\numberwithin{equation}{section}

\dedicatory{Dedicated to the memory of Robert Steinberg}
\begin{document}
\begin{abstract}
We present various results on disconnected reductive groups, in particular
about the characteristic $0$ representation theory of such groups over finite
fields.
\end{abstract}
\maketitle
\section{Introduction}\label{intro}
Let  $\bG$  be  a  (possibly  disconnected)  linear algebraic group over an
algebraically closed field. We assume that the connected component $\Go$ is
reductive,  and then call $\bG$  a (possibly disconnected) reductive group.
This  situation was studied  by Steinberg in  \cite{St} where he introduced
the notion of \qss\ elements.

Assume  now that $\bG$  is over an  algebraic closure $\Fqbar$ of the finite
field  $\Fq$, defined over $\Fq$  with corresponding Frobenius endomorphism
$F$.  Let $\Gun$ be an $F$-stable connected  component of $\bG$. We want to
study  $\GoF$-class functions on  $\GunF$; if $\Gun$  generates $\bG$, they
coincide with $\bG^F$-class functions on $\GunF$.

This  setting we adopt  here is also  taken up by  Lusztig in his series of
papers  on disconnected groups \cite{Lu} and  is slightly more general than
the  setting  of  \cite{grnc},  where  we  assumed that  $\Gun$ contains an
$F$-stable  quasi-central element. A detailed comparison of both situations
is done in the next section.

As  the title says, this  paper is a series  of complements to our original
paper \cite{grnc} which are mostly straightforward developments
that various people  asked us about and, except when mentioned otherwise
(see the introduction to sections \ref{counting} and \ref{shintani descent})
as far as we
know  have not appeared in the  literature; we thank in particular Olivier
Brunat,  Gerhard Hiss,  Cheryl Praeger  and Karine  Sorlin for asking these
questions.

In  section \ref{prel} we show how quite a few results of \cite{grnc}
are  still valid in  our more general  setting.

In section \ref{scalproduct} we take a ``global'' viewpoint to give
a formula for the scalar product of two Deligne-Lusztig characters on the
whole of $\bG^F$.

In section \ref{counting}  we show how  to extend to  disconnected groups the
formula  of Steinberg  \cite[15.1]{St} counting  unipotent elements. 

In section \ref{tensoring} we  extend  the  theorem  that tensoring Lusztig 
induction with the Steinberg character gives ordinary induction.

In  section \ref{characteristic} we  give a formula  for the characteristic
function  of a  \qss\ class, extending the  case of  a quasi-central class
which was treated in \cite{grnc}.

In section \ref{qss classes} we show how to classify \qss\ conjugacy
classes, first for a (possibly disconnected) reductive group over an
arbitrary  algebraically closed field, and then over $\Fq$.

Finally, in section \ref{shintani descent} we  extend to our setting
previous results on Shintani descent.

We thank Gunter Malle for a careful reading of the manuscript.
\section{Preliminaries}\label{prel}

In this paper, we consider a 
(possibly disconnected) algebraic  group $\bG$ over $\Fqbar$
(excepted at the beginning of section \ref{qss classes} where we accept
an arbitrary algebraically closed field),
defined  over  $\Fq$  with  corresponding  Frobenius  endomorphism  $F$. If
$\bG^1$  is an $F$-stable component of  $\bG$, we call  class functions on
$\GunF$ the complex-valued functions invariant under
$\GoF$-conjugacy  (or equivalently under  $\GunF$-conjugacy). Note that
if  $\bG^1$ does not generate $\bG$,  there may be less functions invariant
by  $\bG^F$-conjugacy than by $\GunF$-conjugacy; but the propositions we
prove  will apply {\em in particular} to the $\bG^F$-invariant functions so
we  do  not  lose  any  generality.  The  class functions on $\GunF$ are
provided  with  the  scalar product  $\scal  fg{\GunF}=|\GunF|
\inv\sum_{h\in \GunF}f(h)\overline{  g(h)}$.  We  call  $\bG$  reductive 
when $\bG^0$ is reductive.

When $\bG$ is reductive, following \cite{St} we call \qss\ an element which
normalizes  a pair $\To\subset\Bo$ of a maximal  torus of $\Go$ and a Borel
subgroup  of $\Go$.  Following \cite[1.15]{grnc},  we call  quasi-central a
\qss\   element  $\sigma$  which  has   maximal  dimension  of  centralizer
$C_\Go(\sigma)$  (that we will also  denote by $\bG^{0\,\sigma}$) amongst all \qss\
elements of $\Go\cdot\sigma$.

In the sequel, we fix a reductive group $\bG$
and (excepted in the next section where we take a ``global'' viewpoint)
an $F$-stable connected component $\Gun$ of $\bG$.
In  most of \cite{grnc} we  assumed that 
$\GunF$ contained a quasi-central element. Here we do not assume this.
Note  however that by \cite[1.34]{grnc} $\Gun$ contains an element $\sigma$
which  induces an $F$-stable  quasi-central automorphism of  $\Go$. Such an
element will be enough for our purpose, and we fix one from now on.

By   \cite[1.35]{grnc}   when   $H^1(F,Z\Go)=1$   then   $\GunF$   contains
quasi-central  elements. Here is an example  where $\GunF$ does not contain
quasi-central elements.
\begin{example}\label{not Fstable qc}
Take $s=\begin{pmatrix}\xi&0\\0&1\\  \end{pmatrix}$  
where   $\xi$ is a generator of $\Fq^\times$, take $\Go=\SL_2$ and let
$\bG=\genby{\Go,s}\subset  \GL_2$ endowed with the standard Frobenius
endomorphism  on $\GL_2$, so that $s$ is $F$-stable and $\bG^F=\GL_2(\Fq)$.
We  take  $\Gun=\Go\cdot  s$.  Here  quasi-central elements are central
and coincide with $\Go\cdot s\cap Z\bG$ which is nonempty since if
$\eta\in\BF_{q^2}$ is a square root of $\xi$ then
$\begin{pmatrix}\eta&0\\0&\eta\\  \end{pmatrix}\in \Go\cdot s\cap Z\bG$; but
$\Go\cdot s$ does not meet $(Z\bG)^F$.\hfill$\square$
\end{example}

In the above example $\Gun/\Go$ is a semi-simple element of $\bG/\Go$.
No  such example exists  when $\Gun/\Go$ is  unipotent: 
\begin{lemma}\label{qc rationnel}
Let  $\Gun$  be  an  $F$-stable  connected  component  of  $\bG$  such that
$\Gun/\Go$  is  a  unipotent  element  of  $\bG/\Go$. Then $\GunF$ contains
unipotent quasi-central elements.
\end{lemma}
\begin{proof}
Let  $\To\subset\Bo$ be a pair of an  $F$-stable maximal torus of $\Go$ and
an  $F$-stable  Borel  subgroup  of  $\Go$. Then $N_{\bG^F}(\To\subset\Bo)$
meets   $\GunF$,  since  any  two   $F$-stable  pairs  $\To\subset\Bo$  are
$\GoF$-conjugate.  Let $su$  be the  Jordan decomposition  of an element of
$N_\GunF(\To\subset\Bo)$.   Then  $s\in\Go$   since  $\Gun/\Go$  is
unipotent, and $u$ is $F$-stable, unipotent and still in
$N_\GunF(\To\subset\Bo)$ thus \qss, so is quasi-central by
\cite[1.33]{grnc}.
\end{proof}
Note,  however,  that  there  may  exist  a unipotent quasi-central element
$\sigma$  which is rational  as an automorphism  but such that  there is no
rational element inducing the same automorphism.
\begin{example}
We  give  an  example  in  $\bG=\SL_5\rtimes\genby{\sigma'}$  where
$\Go=\SL_5$  has the standard rational structure  over a finite field $\Fq$
of  characteristic  2  with  $q\equiv  1  \mod  5$  and  $\sigma'$  is  the
automorphism  of $\Go$ given by $g\mapsto J\lexp tg\inv J$ where $J$ is the
antidiagonal matrix with all non-zero entries equal to 1, so that $\sigma'$
stabilizes  the pair  $\To\subset\Bo$ where  $\To$ is  the maximal torus of
diagonal  matrices  and  $\Bo$  the  Borel  subgroup  of  upper  triangular
matrices, hence $\sigma'$ is \qss. Let $t$ be the diagonal
matrix  with entries  $(a,a,a^{-4},a,a)$ where  $a^{q-1}$ is  a non trivial
5-th  root of  unity $\zeta\in\Fq$.  We claim  that $\sigma=t\sigma'$ is as
announced: it is still \qss;
we have  $\sigma^2=t\sigma'(t)=t t\inv=1$  so that  $\sigma$ is
unipotent;  we have  $\lexp F\sigma=\lexp  F t t\inv\sigma=\zeta\sigma$, so
that  $\sigma$ is rational as an  automorphism but not rational. Moreover a
rational  element  inducing  the  same  automorphism  must  be  of the form
$z\sigma$  with $z$ central in $\Go$ and $z\cdot\lexp Fz\inv=\zeta\Id$; but
the  center $Z\Go$ is generated by  $\zeta\Id$ and for any $z=\zeta^k\Id\in
Z\Go$ we have $z\cdot\lexp Fz\inv=\zeta^{k(q-1)}\Id=\Id\ne\zeta\Id$.
\hfill$\square$
\end{example}

As  in \cite{grnc}  we call  ``Levi'' of  $\bG$ a  subgroup $\bL$ of the form
$N_\bG(\bL^0\subset\Po)$   where  $\bL^0$  is  a  Levi  subgroup  of  the
parabolic  subgroup  $\Po$  of  $\Go$.  A  particular case is a ``torus''
$N_\bG(\To,\Bo)$  where  $\To\subset\Bo$  is  a  pair of a maximal
torus  of $\Go$ and a Borel subgroup of $\Go$; note that a ``torus'' meets
all connected components of $\bG$, while (contrary to what is stated
erroneously after \cite[1.4]{grnc}) this may not be the case for
a ``Levi''.

We  call ``Levi''  of $\Gun$  a set  of the form $\Lun=\bL\cap\Gun$ where
$\bL$ is a ``Levi'' of $\bG$ and the intersection is nonempty; note that if
$\Gun$  does not generate $\bG$, there may exist several ``Levis'' of $\bG$
which have same intersection with $\Gun$. Nevertheless $\Lun$ determines
$\Lo$ as the identity component of $\genby\Lun$.

We  assume now that $\Lun$ is an  $F$-stable ``Levi'' of $\Gun$ of the form
$N_\Gun(\bL^0\subset\Po)$.  If $\bU$ is the  unipotent radical of $\Po$, we
define   $\bY_\bU^0=\{x\in\Go\mid  x\inv\cdot\lexp   Fx\in\bU\}$  on  which
$(g,l)\in\bG^F\times\bL^F$  such  that  $gl\in\Go$  acts by $x\mapsto gxl$,
where  $\bL=N_\bG(\Lo,\Po)$. Along  the same  lines as \cite[2.10]{grnc} we
define

\begin{definition}\label{RLG} 
Let $\Lun$ be an $F$-stable ``Levi'' of $\Gun$ of the form
$N_\Gun(\Lo\subset\Po)$ and let $\bU$ be the unipotent radical of $\Po$.
For $\lambda$ a class function on $\LunF$ and
$g\in\GunF$  we set
$$\RLG(\lambda)(g)=|\LunF|\inv\sum_{l\in\LunF}\lambda(l)
\Trace((g,l\inv)\mid H^*_c(\bY_\bU^0))$$
and for $\gamma$ a class function on $\GunF$ and
$l\in\LunF$ we set 
$$\sRLG(\gamma)(l)=|\GunF|\inv\sum_{g\in\GunF}\gamma(g)
\Trace((g\inv,l)\mid H^*_c(\bY_\bU^0)).$$
\end{definition}
In the above $H^*_c$ denotes the $\ell$-adic cohomology with compact support,
where we have chosen once and for all a prime number $\ell\ne p$. In order
to consider the virtual character
$\Trace(x\mid H^*_c(\bX))=\sum_i(-1)^i \Trace(x\mid H^i_c(\bX,\Qlbar))$ as a 
complex character we chose once and for all an embedding
$\Qlbar\hookrightarrow\BC$.

Writing  $\RLG$ and $\sRLG$  is an abuse  of notation: the definition needs
the  choice of a $\Po$ such  that $\Lun=N_\Gun(\Lo\subset\Po)$. 
Our subsequent statements will use an implicit choice. 
Under   certain  assumptions  we  will  prove  a  Mackey  formula  (Theorem
\ref{mackey-levi})  which  when true implies that $\RLG$ and $\sRLG$
are independent of the choice of $\Po$.

By  the  same  arguments  as  for  \cite[2.10]{grnc} (using that $\LunF$ is
nonempty   and  \cite[2.3]{grnc})  definition  \ref{RLG}  agrees  with  the
restriction to $\GunF$ and $\LunF$ of \cite[2.2]{grnc}.

The  two maps  $\RLG$ and  $\sRLG$ are  adjoint with  respect to the scalar
products on $\GunF$ and $\LunF$.

We note the following variation on \cite[2.6]{grnc} where, for
$u$ (resp. $v$) a unipotent element of $\bG$ (resp. $\bL$), we set
$$Q^\Go_\Lo(u,v)=
\begin{cases}
\Trace((u,v)\mid H^*_c(\bY_\bU^0))&\text{if $uv\in\Go$}\\
0&\text{otherwise}\end{cases}.$$
\begin{proposition}\label{char. formula}
Let $su$ be the Jordan decomposition of an element of $\GunF$ and $\lambda$ a
class function on $\LunF$;
\begin{enumerate}
\item if $s$ is central in $\bG$ we have 
$$(R^\Gun_\Lun\lambda)(su)=|\LoF|\inv\sum_{v\in(\Lo\cdot u)^F_\unip}
Q^\Go_\Lo(u,v\inv)\lambda(sv);$$
\item in general
$$(R^\Gun_\Lun\lambda)(su)=
\sum_{\{h\in\GoF\mid \lexp h\bL\owns s\}}
\frac{|\lexp h\Lo\cap\GsoF|}{|\LoF||\GsoF|} 
R_{\lexp h\Lun\cap\Gunsu}^\Gunsu(\lexp h\lambda)(su);
$$
\item if $tv$ is the Jordan decomposition of an element of $\LunF$
and $\gamma$ a class function on $\GunF$, we have
$$(\sRLG\gamma)(tv)=|\GtoF|\inv\sum_{u\in(\Gto\cdot v)^F_\unip}
Q^\Gto_\Lto(u,v\inv)\gamma(tu).$$
\end{enumerate}
\end{proposition}
In the above we abused  notation to
write $\lexp h\bL\owns s$ for $<\Lun>\owns\lexp{h\inv}s$.
\begin{proof}
(i) results from \cite[2.6(i)]{grnc} using the same arguments as
the proof of \cite[2.10]{grnc};
we then get (ii) by plugging back (i) in \cite[2.6(i)]{grnc}.
\end{proof}

In our setting the Mackey formula \cite[3.1]{grnc} 
is still valid in the cases where we proved it
\cite[Th\'eor\`eme 3.2]{grnc} and \cite[Th\'eor\`eme 4.5]{grnc}.
Before stating it notice that \cite[1.40]{grnc} remains true without assuming
that $\GunF$ contains quasi-central elements, replacing in the
proof  $\GoF.\sigma$ with  $\GunF$, which  shows that  any $F$-stable
``Levi''  of $\Gun$ is  $\GoF$-conjugate  to  a  ``Levi''  containing  $\sigma$. This
explains  why we  only state  the Mackey  formula in  the case of ``Levis''
containing $\sigma$.

\begin{theorem}\label{mackey-levi}

If $\Lun$ and $\bM^1$ are two $F$-stable ``Levis'' of $\Gun$ containing $\sigma$
then under one of the following assumptions:
\begin{itemize}
\item $\bL^0$ (resp. $\bM^0$) is a Levi subgroup of an
$F$-stable parabolic subgroup normalized by $\Lun$ (resp. $\bM^1$). 
\item one of $\Lun$ and $\bM^1$ is a ``torus''
\end{itemize}
we have
$$\sRLG R_\Mun^\Gun
=\sum_{x\in[\Lso^F\backslash\cS_\Gso(\Lso,\Mso)^F/\Mso^F]}\hss
R_{(\Lun\cap\lexp x\Mun)}^\Lun\lexp*R_{(\Lun\cap\lexp x\Mun)}
^{\lexp x\Mun}\ad x$$
where $\cS_\Gso(\Lso,\Mso)$ is the set of elements $x\in\Gso$ such that
$\Lso\cap\lexp x \Mso$ contains a maximal torus of $\Gso$.
\end{theorem}

\begin{proof}
We first prove the theorem in the case of $F$-stable parabolic subgroups
$\Po=\bL^0\ltimes\bU$ and $\bQ^0=\bM^0\ltimes\bV$ following the proof
of \cite[3.2]{grnc}.
The difference is that the variety we consider
here is the intersection with $\Go$ of the variety considered in {\it loc.\
cit.}.
Here, the left-hand side of the Mackey formula is given by
$\Qlbar[(\bU^F\backslash \GoF/\bV^F)^\sigma]$ instead of $\Qlbar[(\bU^F\backslash
\GoF.\genby\sigma/\bV^F)^\sigma]$. Nevertheless we can use
\cite[Lemma 3.3]{grnc} which remains valid with the same proof. As for
\cite[Lemma 3.5]{grnc}, we have to replace it with
\begin{lemma}
For any $x\in\cS_\Gso(\Lso,\Mso)^F$ the map
$$(l(\bL^0\cap\lexp x\bV^F),(\lexp x\bM^0\cap\bU^F)\cdot\lexp xm)\mapsto
\bU^Flxm\bV^F$$ is an isomorphism from
$\LoF/(\bL^0\cap\lexp x\bV^F)\times_{(\bL^0\cap\lexp x\bM^0)^F}
(\lexp x\bM^0\cap\bU^F)\backslash\lexp x(\bM^0)^F$ to $
\bU^F\backslash(\Po)^F x(\bQ^0)^F/\bV^F$ which is compatible with the
action of $\LunF\times((\bM^1)^F)\inv$ where the action of
$(\lambda,\mu\inv)\in\LunF\times((\bM^1)^F)\inv$ maps 
$(l(\bL^0\cap\lexp x\bV^F),(\lexp x\bM^0\cap\bU^F)\cdot\lexp xm)$ to
the class of $(\lambda l\nu\inv(\bL^0\cap\lexp x\bV^F),(\lexp
x\bM^0\cap\bU^F)\cdot\nu\lexp xm\mu\inv)$ with $\nu\in \LunF\cap\lexp
x(\bM^1)^F$ (independent of $\nu$).
\end{lemma}
\begin{proof}
The isomorphism of the lemma involves only connected groups and is a known
result (see \emph{e.g.} \cite[5.7]{DM2}). The compatibility with the actions is
straightforward.
\end{proof}
This allows us to complete the proof in the first case.

We now prove the second case following section 4 of \cite{grnc}.
We first notice that the statement and proof of 
Lemma 4.1 in  \cite{grnc} don't use the element $\sigma$ but only its action.
In Lemma 4.2, 4.3 and 4.4 there is no $\sigma$ involved but only the action of the
groups $\bL^F$ and $\bM^F$ on the pieces of a variety depending only on $\bL$,
$\bM$ and the associated parabolics.
This gives the second case.
\end{proof}

We  now  rephrase  \cite[4.8]{grnc}  and  \cite[4.11]{grnc} in our setting,
specializing  the  Mackey  formula  to  the  case  of two ``tori''. 
Let $\cTun$ be the set of ``tori''  of  $\Gun$;   if
$\Tun=N_\Gun(\To,\Bo)\in\cTunF$  then $\To$ is  $F$-stable. 
We define $\Irr(\TunF)$
as the set of restrictions to $\TunF$ of extensions to $\genby\TunF$
of elements of $\Irr(\ToF)$.

\begin{proposition}\label{mackey-tori}
If $\Tun,\Tpun\in\cTunF$ and
$\theta\in\Irr(\TunF),\theta'\in\Irr((\Tpun)^F)$ then
$$\scal{R^\Gun_\Tun(\theta)}{R^\Gun_\Tpun(\theta')}\GunF=0
\text{ unless $(\Tun,\theta)$ and $(\Tpun,\theta')$ are
$\GoF$-conjugate}.$$
And
\begin{itemize}
\item[(i)]
If for some $n\in N_\GoF(\Tun)$ and $\zeta\ne 1$
we have $\lexp n\theta=\zeta\theta$ then
$R^\Gun_\Tun(\theta)=0$.
\item[(ii)]
Otherwise
$\scal{R^\Gun_\Tun(\theta)}{R^\Gun_\Tun(\theta)}\GunF=
|\{n\in N_\GoF(\Tun)\mid \lexp n\theta=\theta\}|/|\TunF|$.
\end{itemize}
If $\Tun=\Tpun$ the above can be written
$$\scal{R^\Gun_\Tun(\theta)}{R^\Gun_\Tun(\theta')}\GunF=
\scal{\Ind_\TunF^{N_\Gun(\To)^F}\theta}{\Ind_\TunF^{N_\Gun(\To)^F}\theta'}
{N_\Gun(\To)^F}$$
where when $A^1\subset B^1$ are cosets of finite groups $A^0\subset B^0$ 
and $\chi$ is a $A^0$-class function on $A^1$ for $x\in B^1$ we
set $\Ind_{A^1}^{B^1}\chi(x)=|A^0|\inv\sum_{\{y\in B^0\mid \lexp yx\in A^1\}}
\chi(\lexp yx)$.
\end{proposition}
\begin{proof}
As noticed above Theorem \ref{mackey-levi} we may assume that $\Tun$ and
$\Tpun$ contain $\sigma$. By \cite[1.39]{grnc}, if $\Tun$ and $\Tpun$ contain
$\sigma$, they are 
$\GoF$ conjugate if and only if they are conjugate under $\Gso^F$.
The Mackey formula shows then that the scalar product vanishes when $\Tun$
and $\Tpun$ are not $\GoF$-conjugate.

Otherwise we may assume $\Tun=\Tpun$ and the Mackey formula gives
$$\scal{R^\Gun_\Tun(\theta)}{R^\Gun_\Tun(\theta)}\GunF=|(\Tsigo)^F|\inv
\sum_{n\in   N_\Gso(\Tsigo)^F}\scal\theta{\lexp  n\theta}\TunF.$$  The
term    $\scal\theta{\lexp   n\theta}\TunF$   is   $0$   unless   $\lexp
n\theta=\zeta_n  \theta$ for some constant $\zeta_n$  and in this last case
$\scal\theta{\lexp     n\theta}\TunF=\overline\zeta_n$.     If    $\lexp
{n'}\theta=\zeta_{n'}    \theta$   then   $\lexp{nn'}\theta=\zeta_{n'}\lexp
n\theta=\zeta_{n'}\zeta_n\theta$  thus the  $\zeta_n$ form  a group;  if
this group is not trivial, that is some $\zeta_n$ is not equal to $1$, we have
$\scal{R^\Gun_\Tun(\theta)}{R^\Gun_\Tun(\theta)}\GunF=0$ which implies
that in this case $R^\Gun_\Tun(\theta)=0$. This gives (i)
since by \cite[1.39]{grnc}, if $\Tun\owns\sigma$
then $N_\GoF(\Tun)=N_\Gso(\Tsigo)^F\cdot\ToF$,
so that if there exists $n$ as in (i) there exists an $n\in
N_\Gso(\Tsigo)^F$ with same action on $\theta$ since $\ToF$ has trivial
action on $\theta$. 

In case (ii), for each non-zero term we have $\lexp n\theta=\theta$ and
we have to check that the value
$|((\bT^\sigma)^0)^F|\inv|\{n\in N_\Gso(\Tsigo)^F\mid \lexp
n\theta=\theta\}|$ given by the Mackey formula is equal to the stated value.
This results again from \cite[1.39]{grnc} written 
$N_\GoF(\Tun)=N_\Gso(\Tun)^F\cdot\ToF$, and 
from $N_\Gso(\Tun)^F\cap\ToF=((\bT^\sigma)^0)^F$.

We now prove the final remark. By definition we have
\begin{multline*}
\scal{\Ind_\TunF^{N_\Gun(\To)^F}\theta}{\Ind_\TunF^{N_\Gun(\To)^F}\theta'}
{N_\Gun(\To)^F}=\\
|N_\Gun(\To)^F|\inv|\TunF|^{-2}\sum_{x\in N_\Gun(\To)^F}
\sum_{\{n,n'\in N_\Gun(\To)^F \mid\lexp nx, \lexp {n'}x\in\Tun\}}
\theta(\lexp nx)\overline{\theta(\lexp{n'}x}).
\end{multline*}
Doing the summation over $t=\lexp nx$ and $n''=n'n\inv \in N_\Go(\To)^F$
we get
$$|N_\Gun(\To)^F|\inv|\TunF|^{-2}\sum_{t\in \TunF}
\sum_{n\in N_\Gun(\To)^F}\sum_{\{n''\in N_\Go(\To)^F\mid\lexp{n''}t\in\Tun\}}
\theta(t)\overline{\theta(\lexp{n''}t}).
$$
The conditions $n''\in N_\Go(\To)^F$ together with $\lexp{n''}t\in\Tun$ are equivalent to 
$n''\in N_\Go(\Tun)^F$, so that we get
$|\TunF|\inv\sum_{n''\in N_\Go(\Tun)^F}
\scal\theta{\lexp{n''}\theta}\TunF$. As explained in the first part of
the proof, the scalar product 
$\scal\theta{\lexp{n''}\theta}\TunF$ is zero unless $\lexp{n''}\theta=
\zeta_{n''}\theta$ for some root of unity $\zeta_{n''}$ and arguing as
in the first part of the proof we find that the above sum is zero if there
exists $n''$ such that $\zeta_{n''}\neq 1$ and is equal to
$|\TunF|\inv |\{n\in N_\GoF(\Tun)\mid \lexp n\theta=\theta\}|$
otherwise.
\end{proof}
\begin{remark} In the context of Proposition \ref{mackey-tori}, if $\sigma$
is  $F$-stable  then  we  may  apply  $\theta$  to  it  and  for  any $n\in
N_\Gso(\Tsigo)^F$ we have $\theta(\lexp n\sigma)=\theta(\sigma)$ so for any
$n\in N_\GoF(\Tun)$  and $\zeta$ such  that $\lexp n\theta=\zeta\theta$  we have $\zeta=1$.
When  $H^1(F,Z\Go)=1$  we may choose $\sigma$ to be $F$-stable, 
so that $\zeta\ne 1$ never happens.

Here  is an example where  $\zeta_n=-1$, thus $R^\Gun_\Tun(\theta)=0$: we
take again the context of Example \ref{not Fstable qc}
and take $\To=\left\{\begin{pmatrix}a&0\\
0&a\inv\\ \end{pmatrix}\right\} $ and let $\Tun=\To\cdot s$; let us define $\theta$
on  $ts\in\TunF$  by  $\theta(t  s)=-\lambda(t)$  where  $\lambda$  is the
non-trivial  order 2 character of  $\ToF$ (Legendre symbol); then for
any $n\in N_\GoF(\Tun)\backslash \To$ we have $\lexp n\theta=-\theta$.
\hfill$\square$\end{remark}
We define {\em uniform} functions as the class functions on $\GunF$ which are
linear combinations of the $R_\Tun^\Gun(\theta)$ for $\theta\in\Irr(\TunF)$.
Proposition \cite[4.11]{grnc} extends as follows to our context:
\begin{corollary}[of \ref{mackey-tori}]\label{(3)}
Let  $p^\Gun$ be  the projector  to uniform  functions on $\GunF$.
We have
$$p^\Gun=|\GunF|\inv\sum_{\Tun\in\cTunF}
|\TunF|R^\Gun_\Tun\circ\lexp*R^\Gun_\Tun.$$
\end{corollary}
\begin{proof}
We have only to check that for any $\theta\in\Irr(\TunF)$ such that
$R^\Gun_\Tun(\theta)\neq 0$ and any class
function $\chi$ on $\GunF$ we have
$\scal{p^\Gun\chi}{R^\Gun_\Tun(\theta)}\GunF=
\scal\chi{R^\Gun_\Tun(\theta)}\GunF$.
By Proposition \ref{mackey-tori}, to evaluate the  
left-hand side we may restrict the sum to tori conjugate to $\Tun$,
so we get 
$$ \begin{aligned}
\scal{p^\Gun\chi}{R^\Gun_\Tun(\theta)}\GunF&=
|N_\GoF(\Tun)|\inv|\TunF|
\scal{R^\Gun_\Tun\circ\lexp*R^\Gun_\Tun\chi}
{R^\Gun_\Tun(\theta)}\GunF\\
&=|N_\GoF(\Tun)|\inv|\TunF|
\scal\chi{R^\Gun_\Tun\circ\lexp*R^\Gun_\Tun\circ
R^\Gun_\Tun(\theta)}\GunF.
\end{aligned}$$
The equality to prove is true if $R^\Gun_\Tun(\theta)=0$; otherwise
by Proposition \ref{mackey-tori} we have
$\lexp*R^\Gun_\Tun\circ R^\Gun_\Tun(\theta)=
|\TunF|\inv\sum_{n\in N_\GoF(\Tun)}\lexp n\theta$, whence in that case
$$R^\Gun_\Tun\circ\lexp*R^\Gun_\Tun\circ
R^\Gun_\Tun(\theta)=
|\TunF|\inv|N_\GoF(\Tun)|R^\Gun_\Tun(\theta),$$ since $R^\Gun_\Tun(\lexp
n\theta)=R^\Gun_\Tun(\theta)$, whence the result.
\end{proof}

We now adapt the definition of duality to our setting.
\begin{definition}
\begin{itemize}
\item
For  a connected  reductive group  $\bG$, we  define the  $\Fq$-rank as the
maximal dimension of a split torus, and define
$\eps_\bG=(-1)^{\text{$\Fq$-rank of $\bG$}}$ and
$\eta_\bG=\eps_{\bG/\rad\bG}$.
\item
For  an $F$-stable connected component  $\Gun$ of a (possibly disconnected)
reductive  group we define  $\eps_\Gun=\eps_\Gso$ and $\eta_\Gun=\eta_\Gso$
where  $\sigma\in\Gun$ induces an
$F$-stable quasi-central automorphism of $\Go$.
\end{itemize}
\end{definition}
Let   us   see   that   these   definitions   agree  with  \cite{grnc}:  in
\cite[3.6(i)]{grnc}, we define $\eps_\Gun$ to be $\eps_{\bG^{0\tau}}$ where
$\tau$  is  any  quasi-semi-simple  element  of  $\Gun$  which  induces  an
$F$-stable  automorphism  of  $\Go$  and  lies  in  a ``torus'' of the form
$N_\Gun(\bT_0\subset\bB_0)$  where both $\To$ and  $\Bo$ are $F$-stable; by
\cite[1.36(ii)]{grnc} a $\sigma$ as above is such a $\tau$.

We fix an $F$-stable pair $(\bT_0\subset\bB_0)$ and
define duality on $\Irr(\GunF)$ by
\begin{equation}\label{duality}
D_\Gun=\sum_{\Po\supset\Bo}\eta_{\bL^1} \RLG\circ\sRLG
\end{equation}
where in the sum $\Po$ runs over $F$-stable parabolic
subgroups containing $\Bo$ such that $N_\Gun(\Po)$ is non
empty, and $\Lun$ denotes $N_\Gun(\bL^0\subset\Po)$ where $\bL^0$ is the
Levi subgroup of $\Po$ containing $\To$.
The duality thus defined coincides with the duality
defined in \cite[3.10]{grnc} when $\sigma$ is in $\GunF$.

In our context we can define $\St_\Gun$ similarly to \cite[3.16]{grnc},
as $D_\Gun(\Id_\Gun)$, and \cite[3.18]{grnc} remains true: 
\begin{proposition}\label{st} $\St_\Gun$ vanishes outside
\qss\ elements, and if $x\in\GunF$ is \qss\ we have 
$$ \St_\Gun(x)=\eps_\Gun\eps_{(\bG^x)^0}|(\bG^x)^0|_p. $$
\end{proposition}

\section{A global formula for the scalar product of Deligne-Lusztig
characters}\label{scalproduct}
In this section we give a result of a different flavor,
where we do not restrict our attention to a connected component $\Gun$.
\begin{definition}\label{RTG}
For  any  character  $\theta$  of  $\bT^F$,  we  define  $R_\bT^\bG$  as in
\cite[2.2]{grnc}.  If for a ``torus''  $\bT$ and $\alpha=g\Go\in\bG/\Go$ we
denote  by  $\bT\comp  \alpha$  or  $\bT\comp  g$  the  unique  connected
component  of $\bT$ which  meets $g\Go$, this  is equivalent 
for $g\in\bG^F$ to 
$$R^\bG_\bT(\theta)(g)=
|\ToF|/|\bT^F|\sum_{\{a\in[\bG^F/\GoF]\mid\lexp ag\in\bT^F\GoF\}}
R^{\bG\comp{\lexp ag}}_{\bT\comp{\lexp ag}}(\theta)(\lexp ag)$$
where the right-hand side is defined by \ref{RLG} (see \cite[2.3]{grnc}).
\end{definition}

We deduce from Proposition \ref{mackey-tori}
the following formula for the whole group $\bG$:
\begin{proposition}\label{scal}
Let $\bT$, $\bT'$ be two ``tori'' of $\bG$ and let $\theta\in\Irr(\bT^F),
\theta'\in\Irr(\bT^{\prime F})$. Then
$\scal{R^\bG_\bT(\theta)}{R^\bG_{\bT'}(\theta')}{\bG^F}=0$
if $\To$ and $\Tpo$ are not $\bG^F$-conjugate, and if 
$\To=\Tpo$ we have
$$\scal{R^\bG_\bT(\theta)}{R^\bG_{\bT'}(\theta')}{\bG^F}=
\scal{\Ind_{\bT^F}^{N_\bG(\To)^F}(\theta)}
{\Ind_{\bT^{\prime F}}^{N_\bG(\To)^F}(\theta')}{N_\bG(\To)^F}.$$
\end{proposition}
\begin{proof}
Definition \ref{RTG} can be written
$$R^\bG_\bT(\theta)(g)=
|\ToF|/|\bT^F|
\sum_{\{a\in[\bG^F/\GoF]\mid\lexp ag\in\bT^F\GoF\}}
R^{\bG\comp g}_{(\lexp{a\inv}\bT)\comp g}(\lexp{a\inv}\theta)(g).$$
So the scalar product we want to compute is equal to
\begin{multline*}\scal{R^\bG_\bT(\theta)}{R^\bG_{\bT'}(\theta')}{\bG^F}=
\frac1{|\bG^F|}\frac{|\ToF|}{|\bT^F|}\frac{|\TpoF|}
{|\bT^{\prime F}|}\\
\sum_{\substack{\alpha\in\bG^F/\Go^F\\g\in\GoF.\alpha}}
\sum_{\substack{\{a\in[\bG^F/\GoF]\mid\lexp a\alpha\in\bT^F\GoF\}\\
\{a'\in[\bG^F/\GoF]\mid\lexp {a'}\alpha\in\bT^{\prime F}\GoF\}}}
R^{\bG.\alpha}_{(\lexp{a\inv}\bT)\comp \alpha}(\lexp{a\inv}\theta)(g)
\overline{R^{\bG.\alpha}_{(\lexp{a^{\prime -1}}\bT')\comp \alpha}(\lexp{a^{\prime
-1}}\theta')(g)},
\end{multline*}
which can be written
\begin{multline*}\scal{R^\bG_\bT(\theta)}{R^\bG_{\bT'}(\theta')}{\bG^F}=
\frac{|\GoF|}{|\bG^F|}\frac{|\ToF|}{|\bT^F|}\frac{|\TpoF|}
{|\bT^{\prime F}|}\\
\sum_{\alpha\in\bG^F/\Go^F}
\sum_{\substack{\{a\in [\bG^F/\GoF]\mid\lexp a\alpha\in\bT^F\GoF\}\\
\{a'\in[\bG^F/\GoF]\mid\lexp {a'}\alpha\in\bT^{\prime F}\GoF\}}}
\scal{R^{\bG.\alpha}_{(\lexp{a\inv}\bT)\comp \alpha}(\lexp{a\inv}\theta)}
{R^{\bG.\alpha}_{(\lexp{a^{\prime -1}}\bT')\comp \alpha}(\lexp{a^{\prime
-1}}\theta')}{\GoF.\alpha}.
\end{multline*}
By Proposition \ref{mackey-tori} the scalar product on the right-hand side is zero
unless $(\lexp{a\inv}\bT)\comp \alpha$ and $(\lexp{a^{\prime -1}}\bT')\comp
\alpha$
are $\GoF$-conjugate, which implies that $\To$ and $\bT^{\prime0}$ are
$\GoF$-conjugate. So we can assume that $\To=\bT^{\prime0}$. Moreover for
each $a'$ indexing a non-zero summand, there is a representative
$y\in a^{\prime-1}\GoF$ such that
$(\lexp y\bT')\comp \alpha= (\lexp{a\inv}\bT)\comp \alpha$. This last equality and
the condition on $a$ imply
the condition $\lexp {a'}\alpha\in\bT^{\prime F}\GoF$ since this condition can
be written $(\lexp y\bT')\comp \alpha\neq\emptyset$.
Thus we can do the summation over all such $y\in\bG^F$, provided
we  divide by
$|N_\GoF((\lexp{a\inv}\bT)\comp \alpha)|$.
So we get, applying Proposition \ref{mackey-tori} that the above
expression is equal to
\begin{multline*}
\frac{|\GoF|}{|\bG^F|}\frac{|\ToF|^2}{|\bT^F|
|\bT^{\prime F}|}
\sum_{\alpha\in\bG^F/\Go^F}
\sum_{\{a\in[\bG^F/\GoF]\mid\lexp a\alpha\in\bT^F\GoF\}}
|N_\GoF((\lexp{a\inv}\bT)\comp \alpha)|\inv\\
\sum_{\{y\in\bG^F\mid (\lexp
y\bT')\comp \alpha=(\lexp{a\inv}\bT)\comp \alpha\}}
\scal{\Ind_{(\lexp{a\inv}\bT)\comp \alpha}^{N_{\Go.\alpha}(\lexp{a\inv}\To)^F}
\lexp{a\inv}\theta}{\Ind_{(\lexp{a\inv}\bT)\comp \alpha}^{N_{\Go.\alpha}
(\lexp{a\inv}\To)^F}
\lexp y\theta'}{N_{\Go.\alpha}(\To)^F}.
\end{multline*}
We now conjugate everything by $a$, take $ay$ as new variable $y$ and set
$b=\lexp a\alpha$. We get
\begin{multline}\label{<RTG,RTG>}
\frac{|\ToF|^2}{|\bT^F||\bT^{\prime F}|}
\sum_{b\in\bT^F/\ToF}|N_\GoF(\bT\comp b)|\inv\\
\sum_{\{y\in\bG^F\mid (\lexp y\bT')\comp b=\bT\comp b\}}
\scal{\Ind_{\bT\compF b}^{N_{\Go.b}(\To)^F}
\theta}{\Ind_{\bT\compF b}^{N_{\Go.b}(\To)^F}\lexp y\theta'}
{N_{\Go.b}(\To)^F},
\end{multline}
since for $b\in\bT^F/\ToF$ any choice of $a\in\bG^F/\GoF$ gives an
$\alpha=\lexp{a\inv}b$ which satisfies the condition  
$\lexp a\alpha\in\bT^F\GoF$.

Let us now transform the right-hand side of \ref{scal}. Using the definition we have
\begin{multline*}
\scal{\Ind_{\bT^F}^{N_\bG(\To)^F}(\theta)}
{\Ind_{\bT^{\prime F}}^{N_\bG(\To)^F}(\theta)}{N_\bG(\To)^F}=\hfill\\
\hfill|\bT^F|\inv|\bT^{\prime F}|\inv|N_\bG(\To)^F|\inv
\sum_{\{n,x,x'\in N_\bG(\To)^F\mid \lexp xn\in\bT,\lexp{x'}n\in\bT'\}}
\theta(\lexp xn)\overline{\theta'(\lexp{x'}n)}=\\
|\bT^F|\inv|\bT^{\prime F}|\inv|N_\bG(\To)^F|\inv\hfill\\
\sum_{b,a,a'\in [N_\bG(\To)^F/N_\Go(\To)^F]}
\sum_{\left\{\substack{n\in N_\Go(\To)^Fb\\x_0,x'_0\in N_\Go(\To)^F}\left\vert 
\substack{
\lexp{x_0}n\in(\lexp{a\inv}\bT)\comp b\\ \lexp{x'_0}n\in(\lexp{a^{\prime-1}}
\bT')\comp b}\right.\right\}}\lexp{a\inv}\theta(\lexp{x_0}n)
\overline{\lexp{a^{\prime-1}}\theta'(\lexp{x'_0}n)}=\\
\frac{|\ToF|}{|\bT^F|}\frac{|\TpoF|}{|\bT^{\prime F}|}
\frac{|N_{\bG^0}(\To)^F|}{|N_\bG(\To)^F|}\hfill\\
\hfill\sum_{b,a,a'\in [N_\bG(\To)^F/N_\Go(\To)^F]}
\scal{\Ind^{N_\Go(\To)^F\cdot b}_{(\lexp{a\inv}\bT)\compF b}\lexp{a\inv}\theta}
{\Ind^{N_\Go(\To)^F\cdot b}_{(\lexp{a^{\prime-1}}\bT')\compF b}
\lexp{a^{\prime-1}}\theta'} {N_\Go(\To)^Fb}.
\end{multline*}
We may simplify the sum by conjugating by $a$ the terms in the scalar product
to get
$$\scal{\Ind^{N_\Go(\To)^F\cdot\lexp ab}_{\bT\compF{\lexp ab}}\theta}
{\Ind^{N_\Go(\To)^F\cdot\lexp ab}_{(\lexp{aa^{\prime-1}}\bT')\compF{\lexp ab}}
\lexp{aa^{\prime-1}}\theta'}{N_\Go(\To)^F\lexp ab}$$
then we may take, given $a$, the conjugate $\lexp ab$ as new variable $b$,
and $aa^{\prime-1}$ as the new variable $a'$ to get
$$
\frac{|\ToF|}{|\bT^F|}\frac{|\TpoF|}{|\bT^{\prime F}|}
\sum_{b,a'\in [\frac{N_\bG(\To)^F}{N_\Go(\To)^F}]}
\scal{\Ind^{N_\Go(\To)^F\cdot b}_{\bT\compF b}\theta}
{\Ind^{N_\Go(\To)^F\cdot b}_{(\lexp{a'}\bT')\compF b}\lexp{a'}\theta'}
{N_\Go(\To)^Fb}.
$$
Now, by Frobenius reciprocity, for the inner scalar product not to vanish,
there must be some element $x\in N_\Go(\To)^F$ such that 
$\lexp x{(\lexp{a'}\bT')\compF b}$ meets $\bT\compF b$ which, considering the
definitions, implies that $(\lexp{xa'}\bT')\comp b=\bT\comp b$.
We may then conjugate the term
$\Ind^{N_\Go(\To)^F\cdot b}_{(\lexp{a'}\bT')\compF b}\lexp{a'}\theta'$ by
such an $x$ to get
$\Ind^{N_\Go(\To)^F\cdot b}_{\bT\compF b}\lexp{xa'}\theta'$ and take $y=xa'$ as a new
variable, provided we count the number of $x$ for a given $a'$, which is
$|N_\Go(\bT\comp b)^F|$. We get
\begin{multline}\label{<Ind,Ind>}
\frac{|\ToF|}{|\bT^F|}\frac{|\TpoF|}{|\bT^{\prime F}|}
\sum_{b\in [N_\bG(\To)^F/N_\Go(\To)^F]}|N_\Go(\bT\comp b)^F|\inv\hfill\\
\hfill\sum_{\{y\in N_\bG(\To)^F\mid(\lexp y\bT')\comp b=\bT\comp b\}}
\scal{\Ind^{N_\Go(\To)^F\cdot b}_{\bT\compF b}\theta}
{\Ind^{N_\Go(\To)^F\cdot b}_{\bT\compF b}\lexp y\theta'}
{N_\Go(\To)^Fb}.
\end{multline}
Since any $b\in[N_\bG(\To)^F/N_\Go(\To)^F]$ such that 
$\bT\compF b$ is not empty has a representative in $\bT^F$
we can do the first summation over $b\in[\bT^F/\ToF]$ so that
\ref{<RTG,RTG>} is equal to \ref{<Ind,Ind>}.
\end{proof}
\section{Counting unipotent elements in disconnected groups}\label{counting}
We wrote the following in february 1994, in answer to a question of Cheryl
Praeger. We are aware that a proof appeared recently in \cite[Theorem
1.1]{LLS} but our original proof reproduced here is much shorter and casefree.
\begin{proposition}\label{nbunip} 
Assume $\Gun/\Go$ unipotent and take
$\sigma\in\Gun$ unipotent $F$-stable and quasi-central (see 
\ref{qc rationnel}). Then
the number of unipotent elements of $\GunF$ is given by
$|(\Gso)^F|_p^2|\Go^F|/|(\Gso)^F|$.
\end{proposition}
\begin{proof}
Let $\chi_\cU$ be the characteristic function of the set of unipotent elements
of $\GunF$. Then $|(\Gun)^F_\unip|=|\GunF|\scal{\chi_\cU}\Id\GunF$ and
$$
\scal{\chi_\cU}\Id\GunF=\scal{\D_\Gun(\chi_\cU)}
{\D_\Gun(\Id)}\GunF=\scal{\D_\Gun(\chi_\cU)}{\St_\Gun}\GunF,
$$
the first equality since $\D_\Gun$ is an isometry by \cite[3.12]{grnc}.
According to \cite[2.11]{grnc}, for any $\sigma$-stable and $F$-stable Levi 
subgroup $\bL^0$ of a $\sigma$-stable parabolic subgroup of $\Go$, setting
$\Lun=\bL^0.\sigma$, we have
$R^\Gun_\Lun(\pi.{\chi_\cU}|_\LunF)=
R^\Gun_\Lun(\pi).{\chi_\cU}$ and
${}^* R^\Gun_\Lun(\varphi).{\chi_\cU}|_\LunF=
{}^* R^\Gun_\Lun(\varphi.{\chi_\cU})$,
thus, by \ref{duality},
$\D_\Gun(\pi.{\chi_\cU})=\D_\Gun(\pi).{\chi_\cU}$;
in particular $\D_\Gun({\chi_\cU})=\D_\Gun(\Id).{\chi_\cU}=\St_\Gun.{\chi_\cU}$.
Now, by Proposition \ref{st}, the only unipotent elements on which $\St_\Gun$
does not vanish are the \qss\ (thus quasi-central) ones;
by \cite[1.37]{grnc} all such are in the $\Go^F$-class of $\sigma$ and,
again by \ref{st} we have $\St_\Gun(\sigma)=|(\Gso)^F|_p$.
We get
$$
\begin{aligned}
|\GunF|\scal{\D_\Gun(\chi_\cU)}{\St_\Gun}\GunF
&=|\GunF|\scal{\St_\Gun.{\chi_\cU}}{\St_\Gun}\GunF\\
&=|\{\hbox{$\Go^F$-class of $\sigma$}\}| |(\Gso)^F|_p^2\\
\end{aligned}$$ whence the proposition.
\end{proof}
\begin{example}
The  formula  of  Proposition  \ref{nbunip}  applies in the following cases
where  $\sigma$ induces a  diagram automorphism of  order $2$ and  $q$ is a
power of $2$:
\begin{itemize}
\item $\Go=\SO_{2n}$, $(\Gso)^F=\SO_{2n-1}(\Fq)$;
\item $\Go=\GL_{2n}$, $(\Gso)^F=\Sp_{2n}(\Fq)$;
\item $\Go=\GL_{2n+1}$, $(\Gso)^F=\SO_{2n+1}(\Fq)\simeq \Sp_{2n}(\Fq)$;
\item $\Go=E_6$, $(\Gso)^F=F_4(\Fq)$;
\end{itemize}

And  it  applies  to  the  case  where  $\Go=\Spin_8$ where $\sigma$ induces a
diagram  automorphism of order $3$ and $q$ is a power of $3$, in which case
$(\Gso)^F=G_2(\Fq)$.
\end{example}

\section{Tensoring by the Steinberg character}\label{tensoring}
\comm{Jean Michel, june 1994}
\begin{proposition}\label{*R}
Let $\Lun$ be an $F$-stable ``Levi'' of $\Gun$.
Then, for any class function $\gamma$ on $\GunF$ we have: $$\lexp
*R^\Gun_\Lun(\gamma\cdot\eps_\Gun\St_\Gun)=\eps_\Lun\St_\Lun
\Res^\GunF_\LunF\gamma.$$
\end{proposition}
\begin{proof}
Let $su$ be the Jordan decomposition of a quasi-semi-simple element of
$\Gun$ with $s$ semi-simple.
We claim  that $u$ is quasi-central in $\bG^s$.
Indeed $su$, being quasi-semi-simple, is in a ``torus'' $\bT$,
thus $s$ and $u$ also are in $\bT$. By \cite[1.8(iii)]{grnc} the
intersection of $\bT\cap\bG^s$ is a ``torus'' of $\bG^s$, thus
$u$ is quasi-semi-simple in $\bG^s$, hence quasi-central since unipotent.

Let $tv$ be the Jordan decomposition of an element $l\in\LunF$ where $t$ is
semi-simple.
Since $\St_\Lun$  vanishes  outside  \qss\  elements the right-hand side of the
proposition vanishes on $l$ unless it is \qss\ which by our claim 
means that $v$ is quasi-central in $\bL^t$.
By  the character formula \ref{char.  formula} the left-hand side of the
proposition evaluates
at    $l$    to    $$\lexp   *R^\Gun_\Lun(\gamma\cdot\eps_\Gun\St_\Gun)(l)=
|\GtoF|\inv\sum_{u\in(\Gto\cdot v)^F_\unip}
Q^\Gto_\Lto(u,v\inv)\gamma(tu)\eps_\Gun\St_\Gun(tu).$$    
By    the    same argument as above, applied to $\St_\Gun$,  the only non
zero   terms  in  the above  sum  are  for  $u$  quasi-central  in  $\bG^t$.
For such $u$, by \cite[4.16]{grnc}, $Q^\Gto_\Lto(u,v\inv)$ vanishes  
unless $u$  and $v$ are
$(\Gto)^F$-conjugate. Hence both sides of the equality to prove vanish
unless $u$ and $v$ are quasi-central and $(\Gto)^F$-conjugate.
In that case  by \cite[4.16]{grnc} and \cite[(**) page  98]{DM2} we have
$Q^\Gto_\Lto(u,v\inv)=Q^\Glo_\Llo(1,1)=
\eps_\Glo\eps_\Llo|\GloF|_{p'}|\LloF|_p$.
Taking  into  account  that  the  $\GtoF$-class of $v$ has
cardinality $|\GtoF|/|\GloF|$ and that by \ref{st} we have
$\St_\Gun(l)=\eps_\Gso\eps_\Glo|\GloF|_p$,  the  left-hand  side of the
proposition reduces to
$\gamma(l)\eps_\Llo|\LloF|_p$,  which is  also the  value of the right-hand
side by applying \ref{st} in $\Lun$.
\end{proof}
By adjunction, we get
\begin{corollary}\label{R}
For any class function $\lambda$ on $\LunF$ we have:
$$R^\Gun_\Lun(\lambda)\eps_\Gun\St_\Gun=
\Ind^\GunF_\LunF(\eps_\Lun\St_\Lun\lambda)$$
\end{corollary}
\section{Characteristic functions of \qss\ classes}\label{characteristic}
\comm{28 oct 1993}

One  of the goals of this section is Proposition \ref{fc qss} where we give
a  formula for the characteristic function of  a \qss\ class which shows in
particular  that it is uniform; this  generalizes the case of quasi-central
elements given in \cite[4.14]{grnc}.

If  $x\in\GunF$  has Jordan decomposition $x=su$ we will denote by $d_x$ 
the map from class functions on  $\GunF$ to class functions on $(\Gunu)^F$ 
given by $$(d_x f)(v)=
\begin{cases}f(sv)&\text{if $v\in(\Gunu)^F$ is unipotent}\\ 0&\text{otherwise.}\\
\end{cases}$$

\begin{lemma}\label{dx_o_R} 
Let $\Lun$ be an $F$-stable ``Levi'' of $\Gun$.
If $x=su$ is the Jordan decomposition of an element of $\LunF$ we have
$d_x\circ\lexp *R^\Gun_\Lun=\lexp* R^\Gunu_\Lunu\circ d_x$.
\end{lemma}
\begin{proof}For $v$ unipotent in $(\Gunu)^F$ and $f$ a class function on
$\GunF$
we have  $$  (d_x\lexp{*}R^\Gun_\Lun f)(v)=(\lexp{*}R^\Gun_\Lun
f)(sv)=  (\lexp  *R^\Gunsu_\Lunsu f)(sv)=(\lexp  *R^\Gunu_\Lunu  d_x
f)(v)$$  where the second equality is \cite[2.9]{grnc}  and the last is by the
character formula  \ref{char. formula}(iii).
\end{proof}

\begin{proposition}\label{d o p}
If $x=su$ is the Jordan decomposition of an element of
$\GunF$, we have $d_x\circ p^\Gun=p^\Gunu\circ d_x$.
\end{proposition}
\begin{proof} Let $f$ be a class function on $\GunF$.
For $v\in(\Gunu)^F$ unipotent, we have, where the last equality is by
\ref{(3)}:
$$(d_xp^\Gun f)(v)=p^\Gun f(sv)=
|\GunF|\inv\sum_{\Tun\in\cTunF}|\TunF|(R^\Gun_\Tun\circ\lexp*R^\Gun_\Tun
f)(sv)
$$
which by Proposition \ref{char. formula}(ii) is:
$$\sum_{\Tun\in\cTunF}
\sum_{\{h\in\GoF\mid\lexp h\bT\owns s\}}
\frac{|\lexp h\To\cap\GsoF|}{|\GoF||\GsoF|} 
(R^\Gunsu_\hTunsu \circ\lexp {h*}R^\Gun_\Tun f)(sv).
$$
Using that $\lexp{h*}R^\Gun_\Tun f=\lexp*R^\Gun_{\lexp h\Tun} f$ and
summing over the $\lexp h\Tun$, this becomes
$$\sum_{\{\Tun\in\cTunF\mid \bT\owns s\}}
\frac{|\To\cap\GsoF|}{|\GsoF|}(R^\Gunsu_\Tunsu\circ\lexp{*}R^\Gun_\Tun
f)(sv).$$
Using that by Proposition \ref{char. formula}(i) for any class function $\chi$ on $\Tunsu^F$ 
$$
\begin{aligned}
(R^\Gunsu_\Tunsu\chi)(sv)&=|\To\cap\GsoF|\inv
\sum_{v'\in(\bT\cap\Gunu)^F_\unip}Q^{(\bG^s)^0}_\Tso(v,v^{\prime-1})
\chi(sv')\\
&=R^\Gunu_{\bT\cap\Gunu}(d_x\chi)(v),
\end{aligned}$$
and using Lemma \ref{dx_o_R}, we get
$$
|\Gunsu^F|\inv
\sum_{\{\Tun\in\cTunF\mid \bT\owns s\}}|\Tso^F|(R^\Gunu_{\bT\cap\Gunu}
\circ\lexp{*}R^\Gunu_{\bT\cap\Gunu}d_xf)(v)
$$
which is the desired result if we apply Corollary \ref{(3)} in $\Gunu$ and 
remark that by \cite[1.8 (iv)]{grnc}  the map
$\Tun\mapsto\bT\cap\Gunu$ induces a bijection between
$\{\bT^1\in\cTunF\mid\bT\owns s\}$ and $F$-stable ``tori'' of $\Gunu$.
\end{proof}
\begin{corollary}\label{uniform} A class function  $f$ on $\GunF$ is uniform
if and only if for every $x\in\GunF$ the function 
$d_xf$ is uniform.\end{corollary}

\begin{proof}Indeed, 
$f=p^\Gun  f$ if and only if for any $x\in \GunF$ we have $d_x f=d_x
p^\Gun f=p^\Gunu d_x f$, the last equality by  Proposition \ref{d o p}.
\end{proof}

For $x\in \GunF$ we consider the class function $\pi_x^\Gun$ on $\GunF$
defined by $$\pi_x^\Gun(y)=\begin{cases}
0 &\text{ if }y \text{ is not conjugate to }x\\
|C_\Go(x)^F|& \text{ if }y=x
\end{cases}$$
\begin{proposition}\label{fc qss}For $x\in\GunF$ quasi-semi-simple 
the function $\pi^\Gun_x$ is uniform, given by
$$\begin{aligned}
\pi^\Gun_x&=\eps_\Gxo|C_\bG(x)^0|_p\inv
\sum_{\{\Tun\in\cTunF\mid \Tun\owns x\}}\eps_\Tun \RTG(\pi_x^\Tun)\\
&=|W^0(x)|\inv\sum_{w\in W^0(x)} \dim R_{\bT_w}^{C_\bG(x)^0}(\Id)
R_{C_\Gun(\bT_w)}^\Gun(\pi_x^{C_\Gun(\bT_w)})
\end{aligned}
$$
where in the second equality $W^0(x)$ denotes the Weyl group of $C_\bG(x)^0$
and $\bT_w$ denotes an $F$-stable torus of type $w$ of this last group.
\end{proposition}
\begin{proof}
First, using Corollary \ref{uniform} we prove that $\pi_x^\Gun$ is uniform. 
Let $su$ be the Jordan decomposition of $x$.
For $y\in\GunF$ the function $d_y\pi_x^\Gun$ is zero unless the
semi-simple part of $y$ is conjugate to $s$. Hence it is sufficient
to evaluate $d_y\pi_x^\Gun(v)$ for elements $y$ whose semi-simple part
is equal to $s$. For such elements
$d_y\pi_x^\Gun(v)$ is up to a coefficient equal to $\pi_u^\Gunu$.
This function is uniform by \cite[4.14]{grnc}, since $u$ being the unipotent 
part of a \qss\ element is quasi-central in $C_\bG(s)$ (see beginning of the
proof of Proposition \ref{*R}).

We have thus $\pi_x^\Gun=p^\Gun\pi_x^\Gun$. We use this to get the formula
of the proposition. We start by using Proposition \ref{st} to write
$\pi_x^\Gun\St_\Gun=\eps_\Gun\eps_\Gxo|(\Gxo)^F|_p\pi_x^\Gun$,
or equivalently $\pi_x^\Gun=\eps_\Gun\eps_\Gxo|(\Gxo)^F|_p\inv
p^\Gun(\pi_x^\Gun\St_\Gun)$.
Using Corollary \ref{(3)} and that by Proposition \ref{*R} we have
$\lexp*\RTG(\pi_x^\Gun\St_\Gun)=
\eps_\Gun\eps_\Tun\St_\Tun\Res_\TunF^\GunF(\pi_x^\Gun)$,
we get 
$$p^\Gun(\pi_x^\Gun\St_\Gun)=\eps_\Gun|\GunF|\inv\sum_{\Tun\in\cTunF}
|\TunF|\eps_\Tun \RTG(\St_\Tun\Res^\GunF_\TunF(\pi_x^\Gun)).$$
The function $\St_\Tun$ is constant equal to 1. Now we have
$$\Res_\TunF^\GunF\pi_x^\Gun=
|\ToF|\inv\sum_{\{g\in\GoF\mid\lexp gx\in\Tun\}}\pi_{\lexp g x}^\Tun.
$$
To see this, do the scalar product with a class function $f$ on $\TunF$:
$$\scal{\Res_\TunF^\GunF\pi_x^\Gun}f\TunF=
\scal{\pi_x^\Gun}{\Ind_\Tun^\Gun f}\GunF
=|\ToF|\inv\sum_{\{g\in\GoF\mid\lexp gx\in\Tun\}}f(\lexp g x).$$
We then get using that $|\ToF|=|\TunF|$
$$p^\Gun(\pi_x^\Gun\St_\Gun)=\eps_\Gun|\GunF|\inv\sum_{\Tun\in\cTunF}
\sum_{\{g\in\GoF\mid\lexp gx\in\Tun\}}\eps_\Tun  \RTG(\pi_{\lexp gx}^\Tun).$$
Taking $\lexp{g\inv}\Tun$ as summation index we get
$$p^\Gun(\pi_x^\Gun\St_\Gun)=\eps_\Gun\sum_{\{\Tun\in\cTunF\mid\Tun\owns x\}}
\eps_\Tun  \RTG(\pi_x^\Tun),$$
hence
$$\pi_x^\Gun=\eps_\Gxo|(\Gxo)^F|_p\inv
\sum_{\{\Tun\in\cTunF\mid\Tun\owns x\}} \eps_\Tun \RTG(\pi_x^\Tun),$$
which is the first equality of the proposition.

For the second equality of the proposition, we first use \cite[1.8 (iii) and
(iv)]{grnc} to sum over tori of $C_\bG(x)^0$: the $\Tun\in\cTunF$ containing
$x$ are in bijection with the maximal tori of $C_\bG(x)^0$ by $\Tun\mapsto
(\Tun^x)^0$ and conversely $\bS\mapsto C_\Gun(\bS)$. This bijection satisfies $\eps_\Tun=\eps_\bS$ by definition of $\eps$. 
 
We then sum over $(C_\bG(x)^0)^F$-conjugacy classes of maximal tori,
which are parameterized by $F$-conjugacy classes of $W^0(x)$. We then have
to multiply by $|(C_\bG(x)^0)^F|/|N_{(C_\bG(x)^0)}(\bS)^F|$ the term
indexed by the class of $\bS$. Then we sum over the elements of $W^0(x)$.
We then have to multiply the term indexed by $w$ by $|C_{W^0(x)}(wF)|/|W^0(x)|$.
Using $|N_{(C_\bG(x)^0)}(\bS)^F|=|\bS^F||C_{W^0(x)}(wF)|$, and the formula
for $\dim R_{\bT_w}^{C_\bG(x)^0}(\Id)$ we get the result.
\end{proof}
\section{Classification of quasi-semi-simple classes}\label{qss classes}
The first items of this section, before \ref{orbitesrat}, apply for
algebraic groups over an arbitrary algebraically closed field $k$.

We  denote by $\cC(\Gun)$ the  set of conjugacy classes  of $\Gun$, that is
the  orbits under $\Go$-conjugacy, and denote by $\Gunqss$ the set of \qss\
classes.
\begin{proposition}\label{Classes}     For    $\Tun\in\cTun$      write
$\Tun=\To\cdot\sigma$ where $\sigma$ is quasi-central. Then $\Gunqss$ is in
bijection  with the set of $N_\Go(\Tun)$-orbits  in $\Tun$, which itself is
in  bijection  with  the  set  of  $W^\sigma$-orbits  in $\cC(\Tun)$, where
$W=N_\Go(\To)/\To$.  We have  $\cC(\Tun)\simeq\Tun/\Ls(\To)$ where
$\Ls$ is the map $t\mapsto t\inv.\lexp\sigma t$.
\end{proposition}
\begin{proof}   
By   definition every  \qss\   element  of  $\Gun$   is  in  some
$\Tun\in\cTun$  and $\cTun$ is a single orbit under $\Go$-conjugacy. It is thus
sufficient   to   find   how   classes   of  $\Gun$  intersect  $\Tun$.  By
\cite[1.13]{grnc} two elements of $\Tun$ are $\Go$-conjugate if and only if
they are conjugate
under   $N_\Go(\To)$.  We can replace $N_\Go(\To)$ by $N_\Go(\Tun)$ since
if  $\lexp  g(\sigma  t)=\sigma  t'$  where  $g\in
N_\Go(\To)$   then  the  image  of  $g$  in  $W$  lies  in  $W^\sigma$.  By
\cite[1.15(iii)]{grnc} elements of $W^\sigma$ have representatives in $\Gso$.
Write  $g=s\dot w$ where  $\dot w$ is  such a representative and $s\in\To$.
Then  $\lexp{s\dot  w}(t\sigma)=\Ls(s\inv)\lexp  w  t\sigma$  whence the
proposition.
\end{proof}
\begin{lemma}\label{TT}
$\To=\Tsigo.\Ls(\To)$.
\end{lemma}
\begin{proof}This is proved in \cite[1.33]{grnc} when $\sigma$ is unipotent
(and  then  the  product  is  direct).  We  proceed similarly to that proof:
$\Tsigo\cap\Ls(\To)$ is finite, since its exponent divides the order
of   $\sigma$  (if  $\sigma(t\inv\lexp\sigma  t)=t\inv\lexp\sigma  t$  then
$(t\inv\lexp\sigma t)^n=t\inv\lexp{\sigma^n}t$ for all $n\ge 1$), and
$\dim(\Tsigo)+\dim(\Ls(\To))=\dim(\To)$   as   the   exact  sequence
$1\to\To^\sigma\to\To\to\Ls(\To)\to 1$ shows,
using that $\dim(\Tsigo)= \dim\To^\sigma$.
\end{proof}
It follows that $\To/\Ls(\To)\simeq
\Tsigo/(\Tsigo\cap\Ls(\To))$;      since  the set   $\cC(\Gso)_{\text{ss}}$
of semi-simple classes of $\Gso$
identifies  with  the  set  of  $W^\sigma$-orbits  on $\Tsigo$ this induces a
surjective map $\cC(\Gso)_{\text{ss}}\to\Gunqss$.

\begin{example}\label{GLn}
We will describe the \qss\ classes of $\Go\cdot\sigma$, where $\Go=\GL_n(k)$
and $\sigma$
is  the  quasi-central  automorphism  given  by  $\sigma(g)=J \lexp t g\inv
J\inv$,  where, if $n$ is even  $J$ is the matrix 
$\begin{pmatrix} 0&-J_0\\J_0&0\\ \end{pmatrix}$ with
$J_0=\begin{pmatrix} 0&&1\\&\addots&\\ 1&&0\\ \end{pmatrix}$
and if $n$ is odd $J$ is the antidiagonal matrix 
$\begin{pmatrix} 0&&1\\&\addots&\\ 1&&0\\ \end{pmatrix}$
(any  outer algebraic automorphism of $\GL_n$ is equal to $\sigma$ up to an
inner automorphism).

The  automorphism $\sigma$ normalizes the  pair $\To\subset\Bo$ where $\To$
is  the diagonal torus and $\Bo$ the group of upper triangular matrices.
Then $\Tun=N_\Gun(\To\subset\Bo)=\To\cdot\sigma$ is in $\cTun$. For
$\diag(x_1,\ldots,x_n)\in  \To$,  where
$x_i\in k^\times$, we have $\sigma(\diag(x_1,\ldots,x_n))=
\diag(x_n\inv,\ldots,x_1\inv)$. It follows that $\Ls(\To)=
\{\diag(x_1,x_2,\ldots,x_2,x_1)\}$   --- here  $x_{m+1}$  is   a  square  when
$n=2m+1$ but this is not a condition since $k$ is algebraically closed. As
suggested above, we could take as representatives of $\To/\Ls(\To)$ the set
$\Tsigo/(\Tsigo\cap\Ls(\To))$,   but  since   $\Tsigo\cap\Ls(\To)$  is  not
trivial  (it consists of the diagonal matrices with entries $\pm 1$ placed
symmetrically), it is more convenient to take for representatives of the \qss\
classes the set
$\{\diag(x_1,x_2,\ldots,x_{\lfloor\frac n2\rfloor},1,\ldots,1)\}\sigma$. In this
model the action   of  $W^\sigma$   is generated by the permutations of the 
$\lfloor\frac n2\rfloor$ first entries, and by the maps 
$x_i\mapsto x_i\inv$, so the \qss\ classes
of $\Go\cdot\sigma$ are parameterized by the \qss\ classes of $\Gso$.

We continue the example, computing group of components of centralizers.
\begin{proposition}   Let   $s\sigma=\diag(x_1,x_2,\ldots,x_{\lfloor\frac
n2\rfloor},1,\ldots,1)\sigma$  be a \qss\ element  as above. If 
$\text{char }k=2$ then $C_\Go(s\sigma)$ is connected. Otherwise, if $n$
is  odd,  $A(s\sigma):=C_\Go(s\sigma)  /C_\Go(s\sigma)^0$  is of order two,
generated  by $-1\in Z\Go=Z\GL_n(k)$.  If $n$ is  even, $A(s\sigma)\ne 1$ if and
only if for some $i$ we have $x_i=-1$; then $x_i\mapsto x_i\inv$ is an element
of  $W^\sigma$ which  has a  representative in  $C_\Go(s\sigma)$ generating
$A(s\sigma)$, which is of order 2.
\end{proposition}
\begin{proof}
We will use that for a group $G$ and an automorphism 
$\sigma$ of $G$ we have an exact sequence (see for example \cite[4.5]{St})
\begin{equation}\label{exact seq}
1\rightarrow (ZG)^\sigma\rightarrow G^\sigma\rightarrow
(G/ZG)^\sigma\rightarrow (\Ls(G)\cap ZG)/\Ls(ZG)\rightarrow 1
\end{equation}
If we take $G=\Go=\GL_n(k)$ in \ref{exact seq} and $s\sigma$ for $\sigma$,
since on $Z\Go$ the map $\Ls=\cL_{s\sigma}$ is $z\mapsto z^2$,
hence surjective, we get
that $\Go^{s\sigma}\rightarrow \PGL_n^{s\sigma}$ is surjective and
has kernel $(Z\Go)^\sigma=\{\pm1\}$.

Assume $n$ odd and take $G=\SL_n(k)$ in \ref{exact seq}. We have
$Z\SL_n^\sigma=\{1\}$ so that we get the following diagram with exact rows:
$$\xymatrix{1\ar[r]&\{\pm1\}\ar[r]&\GL_n^{s\sigma}\ar[r]&\PGL_n^{s\sigma}\ar[r]&1\\
&1\ar[r]&\SL_n^{s\sigma}\ar@{^{(}->}[u]\ar[r]&\PGL_n^{s\sigma}\ar[r]\ar@{=}[u]&1
}$$
This shows that $\GL_n^{s\sigma}/\SL_n^{s\sigma}\simeq\{\pm 1\}$;
by \cite[8.1]{St} $\SL_n^{s\sigma}$ is connected, hence $\PGL_n^{s\sigma}$ 
is connected thus
$\GL_n^{s\sigma}=(\GL_n^{s\sigma})^0\times \{\pm1\}$
is connected if and only if $\text{char }k =2$.

Assume now $n$ even; then $(\To)^\sigma$ is connected hence
$-1\in(\GL_n^{s\sigma})^0$ for all $s\in\To$.
Using this, the exact sequence
$1\rightarrow\{\pm1\}\rightarrow\GL_n^{s\sigma}\rightarrow\PGL_n^{s\sigma}
\rightarrow1$ implies 
$A(s\sigma)=\bG^{s\sigma}/\Go^{s\sigma}=
\GL_n^{s\sigma}/(\GL_n^{s\sigma})^0\simeq
\PGL_n^{s\sigma}/(\PGL_n^{s\sigma})^0$.
To compute this group we use
\ref{exact seq} with $\SL_n(k)$ for $G$ and $s\sigma$ for $\sigma$:
$$ 1\rightarrow \{\pm 1\}\rightarrow \SL_n^{s\sigma}\rightarrow
\PGL_n^{s\sigma}\rightarrow (\cL_{s\sigma}(\SL_n)\cap Z\SL_n)/\Ls(Z\SL_n)
\rightarrow 1 $$
which, since $\SL_n^{s\sigma}$ is connected, implies that
$A(s\sigma)=(\cL_{s\sigma}(\SL_n)\cap Z\SL_n)/\Ls(Z\SL_n)$ thus is
non trivial (of order 2) if and only if
$\cL_{s\sigma}(\SL_n)\cap Z\SL_n$ contains an element which is not a square
in $Z\SL_n$;
thus $A(s\sigma)$ is trivial if $\text{char }k=2$. We assume now
$\text{char }k\ne 2$. Then a non-square
is of the form $\diag(z,\ldots,z)$ with $z^m=-1$ if we set $m=n/2$.

The following lemma is a transcription of \cite[9.5]{St}.
\begin{lemma}Let  $\sigma$ be a quasi-central automorphism of the connected
reductive  group  $\bG$  which  stabilizes  the  pair  $\bT\subset\bB$ of a
maximal  torus and a Borel subgroup; let $W$ be the Weyl group of $\bT$ and
let   $s\in\bT$.  Then  $\bT\cap\cL_{s\sigma}(\bG)=\{\cL_w(s\inv)\mid
w\in W^\sigma\}\cdot \Ls(\bT)$.
\end{lemma}
\begin{proof}
Assume $t=\cL_{s\sigma}(x)$ for $t\in\bT$, or equivalently
$xt=\lexp{s\sigma}x$. Then if $x$ is in the Bruhat cell
$\bB w\bB$, we must have $w\in W^\sigma$. Taking for $w$ a $\sigma$-stable
representative $\dot w$ and writing the unique Bruhat decomposition
$x=u_1\dot w t_1 u_2$ where $u_2\in \bU, t_1\in\bT$ and 
$u_1\in\bU\cap\lexp w\bU^-$ where $\bU$ is the unipotent radical of $\bB$ and
$\bU^-$ the unipotent radical of the opposite Borel, the equality
$xt=\lexp{s\sigma}x$ implies that $\dot w t_1 t=\lexp{s\sigma}(\dot w t_1)$
or equivalently $t=\cL_{w\inv}(s\inv)\Ls(t_1)$, whence the lemma.
\end{proof}
We apply this lemma taking $\SL_n$ for $\bG$ and $\Tpo=\To\cap\SL_n$ for
$\bT$: we get $\cL_{s\sigma}(\SL_n)\cap Z\SL_n=
\{\cL_w(s\inv)\mid w\in W^\sigma\}\cdot \Ls(\Tpo)\cap Z\SL_n$.
The element $\diag(x_1,x_2,\ldots,x_m,1,\ldots,1)\sigma$
is conjugate to $s\sigma=
\diag(y_1,y_2,\ldots,y_m,y_m\inv,\ldots,y_1\inv)\sigma\in(\Tpo)^\sigma\cdot\sigma$ 
where $y_i^2=x_i$.
It will have a non connected centralizer if and only if for some $w\in
W^\sigma$ and some $t\in\Tpo$ we have
$\cL_w(s\inv)\cdot \Ls(t)=\diag(z,\ldots,z)$ with $z^m=-1$ and then
an appropriate representative of $w$ (multiplying if needed by an
element of $Z\GL_n$) will be in $C_\Go(s\sigma)$ and have a non-trivial image
in $A(s\sigma)$.
Since $s$ and $w$ are $\sigma$-fixed, we have $\cL_w(s)\in(\Tpo)^\sigma$, thus
it is of the form $\diag(a_1,\ldots,a_m,a_m\inv,\ldots a_1\inv)$. Since
$\Ls(\Tpo)=\{\diag(t_1,\ldots,t_m,t_m\ldots,t_1)\mid t_1t_2\ldots t_m=1\}$, we 
get $z=a_1t_1=a_2t_2=\ldots=a_mt_m=a_m\inv t_m=\ldots=a_1\inv t_1$; in
particular $a_i=\pm1$ for all $i$ and $a_1a_2\ldots a_m=-1$.
We can take $w$ up to conjugacy in $W^\sigma$ since 
$\cL_{vwv\inv}(s\inv)=\lexp v\cL_w(\lexp{v\inv}s\inv)$ and $\Ls(\Tpo)$ is 
invariant under $W^\sigma$-conjugacy. We see $W^\sigma$ as the group of 
permutations of $\{1,2,\ldots,m,-m,\ldots,-1\}$ which preserves the pairs
$\{i,-i\}$. A non-trivial cycle of $w$ has, up
to conjugacy, the form either $(1,-1)$ or 
$(1,-2,3,\ldots,(-1)^{i-1}i,-(i+1),-(i+2),\ldots,-k,-1,2,-3,\ldots,k)$
with $0\leq i\leq k\leq n$ and $i$ odd, or
$(1,-2,3,\ldots,(-1)^{i-1}i,i+1,i+2,\ldots,k)$ with $0\leq i\leq k\leq n$
and $i$ even (the case $i=0$ meaning that there is no sign change).
The contribution to $a_1\ldots a_m$ of
the orbit $(1,-1)$ is $a_1=y_1^2$ hence is 1 except if $y_1^2=x_1=-1$. 
Let us consider an orbit of the second
form. The $k$ first coordinates of $\cL_w(s\inv)$
are $(y_1y_2,\ldots,y_iy_{i+1},y_{i+1}/y_{i+2},\ldots,y_k/y_1)$.
Hence there must exist signs $\varepsilon_j$ such that
$y_2=\varepsilon_1/y_1$,
$y_3=\varepsilon_2/y_2$, \dots, $y_{i+1}=\varepsilon_i/y_i$ and
$y_{i+2}=\varepsilon_{i+1}y_{i+1}$,\dots, $y_k=\varepsilon_{k-1}y_{k-1}$,
$y_1=\varepsilon_ky_k$. 
This gives $y_1=\begin{cases}\varepsilon_1\ldots\varepsilon_k y_1&\text{if
$i$ is even}\\ \varepsilon_1\ldots\varepsilon_k/y_1&\text{if
$i$ is odd}\end{cases}.$
The contribution of the orbit to $a_1\ldots a_m$ is $\varepsilon_1\ldots
\varepsilon_k$ thus is $1$ if $i$ is even and $x_1=y_1^2$ if $i$ is odd.
Again, we see that one of the $x_i$ must equal $-1$ to get
$a_1\ldots a_m=-1$. Conversely if $x_1=-1$, for any $z$ such that $z^m=-1$, 
choosing $t$ such that $\Ls(t)=\diag(-z,z,z,\ldots,z,-z)$ and
taking $w=(1,-1)$ we get
$\cL_w(s\inv)\Ls(t)=\diag(z,\ldots,z)$ as desired.
\end{proof}
\end{example}

We now go back to the case where $k=\Fqbar$, and
in the context of Proposition \ref{Classes}, 
we now assume that $\Tun$ is $F$-stable and that $\sigma$ induces an $F$-stable
automorphism of $\Go$.
\begin{proposition}\label{orbitesrat} Let $\Tunrat= \{s\in
\bT^1\mid\exists   n\in  N_\Go(\bT^1),   \lexp{nF}s=s\}$; then
$\Tunrat$ is stable by $\To$-conjugacy, which gives a meaning to
$\cC(\Tunrat)$. Then  $c\mapsto
c\cap\Tun$ induces a bijection between  $(\Gunqss)^F$ 
and the $W^\sigma$-orbits on $\cC(\Tunrat)$.
\end{proposition}
\begin{proof}  A class  $c\in\Gunqss$ is  $F$-stable if  and only  if given
$s\in  c$ we  have $\lexp  Fs\in c$.  If we  take $s\in  c\cap\Tun$ then 
$\lexp Fs\in c\cap\Tun$ which as
observed  in the proof of \ref{Classes} implies that $\lexp Fs$ is conjugate
to  $s$ under $N_\Go(\Tun)$, that is  $s\in\Tunrat$. Thus $c$ is $F$-stable
if  and only if $c\cap\Tun=c\cap\Tunrat$. The proposition then results from
Proposition   \ref{Classes}  observing  that   $\Tunrat$  is  stable  under
$N_\Go(\Tun)$-conjugacy   and  that   the  corresponding   orbits  are  the
$W^\sigma$-orbits on $\cC(\Tunrat)$.
\end{proof}
\begin{example}
When $\Gun=\GL_n(\Fqbar)\cdot\sigma$ with $\sigma$ as in Example \ref{GLn},
the map $$\diag(x_1,x_2,\ldots,x_{\lfloor\frac n2\rfloor},
1,\ldots,1)\mapsto \diag(x_1,x_2,\ldots,x_{\lfloor\frac n2\rfloor},
\dag,x_{\lfloor\frac n2\rfloor}\inv,\ldots,x_2\inv,x_1\inv)$$ where $\dag$
represents 1 if $n$ is odd and an omitted entry otherwise, is compatible with
the action of $W^\sigma$ as described in \ref{GLn} on the left-hand side
and the natural action on the right-hand side. 
This map induces a bijection from  $\Gunqss$  to the 
semi-simple classes of $(\GL_n^\sigma)^0$ which restricts
to a bijection from  $(\Gunqss)^F$  to the $F$-stable
semi-simple classes of $(\GL_n^\sigma)^0$.
\end{example}
We now compute the cardinality of $(\Gunqss)^F$.
\begin{proposition}\label{f  sur qss} Let $f$  be a function on $(\Gunqss)^F$.
Then   $$\sum_{c\in(\Gunqss)^F}f(c)=|W^\sigma|\inv\sum_{w\in   W^\sigma}\tilde
f(w)$$ where $\tilde f(w):=\sum_sf(s)$, where
$s$ runs over representatives in $\Tun^{wF}$ of $\Tun^{wF}/\Ls(\To)^{wF}$.
\end{proposition}
\begin{proof}  
We have 
$\cC(\Tunrat)=\bigcup_{w\in W^\sigma}\{s\Ls(\To)\in\Tun/\Ls(\To)\mid s\Ls(\To)
\text{ is $wF$-stable}\}$.
The conjugation by $v\in W^\sigma$ sends a $wF$-stable coset $s\Ls(\To)$
to a $vwFv\inv$-stable coset; and the number of $w$ such that $s\Ls(\To)$
is $wF$-stable is equal to $N_{W^\sigma}(s\Ls(\To))$.
It follows that   $$\sum_{c\in  (\Gunqss)^F}f(c)=|W^\sigma|\inv   \sum_{w\in
W^\sigma}\sum_{s\Ls(\To)\in(\Tun/\Ls(\To))^{wF}}
f(s\Ls(\To)).$$ 
The proposition follows since, 
$\Ls(\To)$ being connected, we have
$(\Tun/\Ls(\To))^{wF}=\Tun^{wF}/\Ls(\To)^{wF}$.
\end{proof}
\begin{corollary}\label{nb   qss}We have
$|(\Gunqss)^F|=|(\cC(\Gso)_{\text{ss}})^F|$.
\end{corollary}
\begin{proof} Let us take $f=1$ in \ref{f sur qss}. We need to sum
over $w\in W^\sigma$ the value $|\Tun^{wF}/\Ls(\To)^{wF}|$.
First note that $|\Tun^{wF}/\Ls(\To)^{wF}|=|\To^{wF}/\Ls(\To)^{wF}|$.
By Lemma \ref{TT} we have the exact sequence
$$1\to\Tsigo\cap\Ls(\To)\to\Tsigo\times\Ls(\To)\to\To\to 1$$
whence the Galois cohomology exact sequence:
\begin{multline*}
1\to(\Tsigo\cap\Ls(\To))^{wF}\to\Tsigo^{wF}\times(\Ls(\To))^{wF}
\to\hfill\\
\hfill\To^{wF}\to  H^1(wF,(\Tsigo\cap\Ls(\To)))\to  1.
\end{multline*}
Using that for any automorphism $\tau$ of a finite
group $G$ we have $|G^\tau|=|H^1(\tau,G)|$,
we have $|(\Tsigo\cap\Ls(\To))^{wF}|=
|H^1(wF,(\Tsigo\cap\Ls(\To)))|$. Together with the above exact sequence
this implies that
$|\To^{wF}/\Ls(\To)^{wF}|=|\Tsigo^{wF}|$ whence
$$|(\Gunqss)^F|=|W^\sigma|\inv\sum_{w\in    W^\sigma}|\Tsigo^{wF}|.$$ 
The corollary follows by either applying the same formula for the connected
group $\Gso$, or referring to \cite[Proposition 2.1]{Le}.
\end{proof}

\comm{
Soit  $\cT^0$ la  vari\'et\'e des  couples $\bT\subset\bB$  o\`u $\bT$ est un
tore  maximal de  $\Go$ et  $\bB$ est  un sous-groupe  de Borel de $\Go$ le
contenant  (on  a  $\cT^0\simeq\Go/\bT$).  L'application  $\pi:\cT^0\to\cTun:
(\bT,\bB)\mapsto  N_\Gun(\bT,\bB)$ est un  rev\^etement de groupe $W^\sigma$,
donc on a une d\'ecomposition:
$\pi_*(\Qlbar)=\bigoplus_{\chi\in\Irr(W^\sigma)}\cF_\chi$.    Nous   fixons
$\bT^1_0\in\cTun$  et noterons $\Tun\mapsto  w(\Tun)$ l'application naturelle
$\cTunF\to  W^\sigma/\sim F$ qui  associe \`a un  \'el\'ement de $\cT1$ son
``type''  relatif  \`a  $\bT^1_0$  (l'\'el\'ement  $x\inv\lexp Fx$ o\`u $x$
conjugue  $\Tun$ sur $\bT^1_0$ --- deux ``tores'' ont m\^eme ``type'' si et
seulement si ils sont dans la m\^eme $\GoF$-orbite). On a $\Trace(F\mid
(\cF_\chi)_\Tun)=\chi(w(\Tun))$.
\begin{proposition}\label{faisceau} Si $x\in\GunF$, alors
$\sum_{\{\Tun\in \cT1\mid\Tun\owns x\}}\Trace(F\mid (\cF_\chi)_\Tun)=
R^\Gun_{\chi\otimes\varepsilon}\otimes\St_\Gun(x)$.
\end{proposition}
\begin{proof}   En  inversant  sur  $W^\sigma$  cela  revient  \`a  voir  que
$$|C_{W^\sigma}(w)|\#\{\Tun\owns  x\mid  \Tun\text{  est  de  type  }  w\}=
(R_w^\Gun\otimes\St_\Gun)(x)\varepsilon(w)=\Ind_{\bT_w.\sigma}^\Gun(\Id)
(x),$$  o\`u la derni\`ere  \'egalit\'e est par  \ref{R}. Cette \'egalit\'e est
imm\'ediate \`a v\'erifier en utilisant la d\'efinition de l'induit.
\end{proof}}

\section{Shintani descent}\label{shintani descent}
\comm{3rd of july, 1995}

We  now look at  Shintani descent in  our context; we  will show it commutes
with Lusztig induction when $\Gun/\Go$ is semi-simple and the characteristic
is good for $\Gso$. We  should mention previous work on
this  subject: Eftekhari (\cite[II. 3.4]{eftekhari}) has the same result for
Lusztig induction from a torus; he does not need  to assume $p$ good but needs $q$ to
be  large enough  to apply  results of  Lusztig identifying Deligne-Lusztig
induction  with induction  of character  sheaves; Digne (\cite[1.1]{digne})
has the result in the same generality as here apart from the assumption
that  $\Gun$ contains an $F$-stable quasi-central element; however a defect
of  his  proof  is  the  use  without  proof of the property given in Lemma
\ref{good} below.

As above $\Gun$ denotes an $F$-stable
connected  component of  $\bG$ of  the form $\Go\cdot\sigma$ where
$\sigma$ induces a quasi-central automorphism of $\Go$ commuting with $F$.

Applying   Lang's  theorem,  one  can  write   any  element  of  $\Gun$  as
$x\cdot\lexp{\sigma  F}x\inv\sigma$ for some $x\in\Go$, or as $\sigma
\cdot\lexp Fx\inv\cdot  x$ for some $x\in\Go$. Using  that $\sigma$, as
automorphism, commutes with $F$,  it is easy to  check that the
correspondence $x\cdot\lexp{\sigma
F}x\inv\sigma\mapsto \sigma\lexp Fx\inv\cdot x$ induces a bijection $\NFsF$
from  the $\GoF$-conjugacy  classes of  $\GunF$ to the $\Go^\sF$-conjugacy
classes of $(\Gun)^{\sigma F}$ and that $|\Go^\sF||c|=|\GoF||\NFsF(c)|$ for
any $\GoF$-class $c$ in $\GunF$. It follows that the operator $\ShFsF$
from $\GoF$-class functions on  $\GunF$ to $\Go^\sF$-class functions
on $(\Gun)^{\sigma F}$ defined by $\ShFsF(\chi)(\NFsF x)=\chi(x)$ is an
isometry.

The remainder of this section is devoted to the proof of the following
\begin{proposition}\label{Sh o sRLG}
Let $\Lun=N_\Gun(\Lo\subset\Po)$ be a ``Levi'' of $\Gun$ containing
$\sigma$, where $\Lo$ is $F$-stable; we have $\Lun=\Lo\cdot\sigma$.
Assume that $\sigma$ is semi-simple and that the characteristic is good for 
$\Gso$. Then
$$\ShFsF\circ\sRLG=\sRLG\circ\ShFsF \quad\text{ and }\quad
\ShFsF\circ\RLG=\RLG\circ\ShFsF.
$$
\end{proposition}
\begin{proof}
The second equality follows from the first by adjunction, using that the
adjoint of $\ShFsF$ is $\ShFsF\inv$.
Let us prove the first equality.

Let $\chi$ be a $\GoF$-class function on $\Gun$ and let
$\sll u=u\sll$ be the Jordan decomposition of an element of $(\Lun)^{\sigma F}$
with $u$ unipotent and $\sigma l$ semi-simple.
By the character formula \ref{char. formula}(iii) and the definition of
$Q^\Gto_\Lto$ for $t=\sigma l$ we have
$$\displaylines{(\sRLG\ShFsF(\chi))(\sll u)=\hfill\cr\hfill
|\GlsosF|\inv\sum_{v\in\GlsosF_\unip}\ShFsF(\chi)(\sll v)
\Trace((v,u\inv)|H^*_c(\YUsF)),\cr}$$
where   $v$ (resp.\ $u$)  acts  by   left- (resp.\ right-) translation  on  $\YUsF=\{x\in(\bG^\sll)^0\mid
x\inv\cdot\lexp{\sigma F}x\in\bU\}$ where $\bU$ denotes the unipotent radical of
$\Po$;  in the summation $v$ is in the identity component of $\bG^\sll$
since, $\sigma$ being semi-simple, $u$ is in $\Go$ hence in
$(\bG^\sll)^0$ by \cite[1.8 (i)]{grnc} since $\sll$ is semi-simple.

Let us write $l=\lexp F\lambda\inv\cdot\lambda$ with $\lambda\in\Lo$, so that
$\sll=\NFsF(\lps)$ where $l'=\lambda\cdot\lexp\sF\lambda\inv$.
\begin{lemma}\label{v'} For $v\in\GlsosF_\unip$
we have $\sll v=\NFsF((\sll\cdot v')^\sFti)$ where
$v'=n_{\sF/\sF}v\in\GlsosF$ is defined by writing 
$v=\lexp\sF\eta\cdot\eta\inv$ where
$\eta\in(\bG^\st)^0$ and setting $v'=\eta\inv\cdot\lexp\sF\eta$.
\end{lemma}
\begin{proof}
We have $\sll v=\sll\lexp\sF\eta\cdot\eta\inv=\lexp\sF\eta\sll\eta\inv=
\sigma\lexp F(\eta\lambda\inv)\lambda\eta\inv$, thus 
$\sll v=\NFsF((\lambda\eta\inv)\cdot\lexp\sF(\eta\lambda\inv)\sigma)$.
And we have $(\lambda\eta\inv)\cdot\lexp\sF(\eta\lambda\inv)\sigma=
\lambda v'\lexp\sF\lambda\inv\sigma=
\lexp F\lambda l v'\sigma\lexp F\lambda\inv=(\sll v')^\sFti$, thus
$\ShFsF(\chi)(\sll v)=\chi((\sll v')^\sFti)$.
\end{proof}

\begin{lemma}\label{Sh vlsigma} 
\begin{enumerate}
\item We have $(\sll)^\sFti=\lps$.
\item The conjugation $x\mapsto x^\sFti$ maps
$\bG^\sll$ and the action of $\sigma F$ on it, to $\bG^\lps$ with the action
of $F$ on it; in particular it
induces bijections 
$\GlsosF\xrightarrow\sim\GlpsosF$ and
$\YUsF\xrightarrow\sim\YUF$, where
$\YUF=\{x\in(\bG^\lps)^0\mid x\inv\lexp Fx\in\bU\}$.
\end{enumerate}
\end{lemma}
\begin{proof}
(i) is an obvious computation and shows that if $x\in\bG^\sll$ then
$x^\sFti\in\bG^\lps$. To prove (ii), it remains to show that if
$x\in\bG^\sll$ then $\lexp F(x^\sFti)=(\lexp{\sigma F}x)^\sFti$.
From $x^\sigma=x^{l\inv}=x^{\lambda\inv\cdot\lexp F\lambda}$, we get 
$x^\sFti=x^{\lambda\inv}$, whence $\lexp F(x^\sFti)=(\lexp F x)^{\lexp
F\lambda\inv}
=((\lexp{\sigma F}x)^\sigma)^{\lexp F\lambda\inv} =(\lexp{\sigma F}x)^\sFti$.
\end{proof}
Applying lemmas \ref{v'} and \ref{Sh vlsigma} we get
$$\displaylines{(\sRLG\ShFsF(\chi))(\sigma lu)=\hfill\cr\hfill
|\GlsosF|\inv\sum_{v\in\GlsosF_\unip}
\chi((\sll v')^\sFti)
\Trace((v^\sFti,(u^\sFti)\inv)\mid H^*_c(\YUF)).\cr}$$
\begin{lemma}\label{good}
Assume that the characteristic is good for $\Gso$, where $\sigma$ is a
quasi-central element of $\bG$. Then it is also good for
$(\bG^s)^0$ where $s$ is any \qss\ element of $\Go\cdot\sigma$.
\end{lemma}
\begin{proof}  Let $\Sigma_\sigma$ (resp.\ $\Sigma_s$) be the root system of
$\Gso$ (resp.\ $(\bG^s)^0$). By definition, a characteristic $p$ is good for
a  reductive  group  if  for  no  closed  subsystem  of its root system the
quotient  of the generated lattices  has $p$-torsion. The system $\Sigma_s$
is   not  a  closed  subsystem  of  $\Sigma_\sigma$  in  general,  but  the
relationship is expounded in \cite{qss}: let $\Sigma$ be the root system of
$\Go$  with respect to a $\sigma$-stable  pair $\bT\subset\bB$ of a maximal
torus  and a Borel subgroup  of $\Go$. Up to  conjugacy, we may assume that
$s$ also stabilizes that pair. Let $\overline\Sigma$ the set of sums of the
$\sigma$-orbits  in $\Sigma$, and $\Sigma'$ the set of averages of the same
orbits.  Then  $\Sigma'$  is  a  non-necessarily  reduced  root system, but
$\Sigma_\sigma$ and $\Sigma_s$ are subsystems of $\Sigma'$ and are reduced.
The  system $\overline\Sigma$ is reduced, and the set of sums of orbits
whose  average  is  in  $\Sigma_\sigma$ (resp.\ $\Sigma_s$) is a
closed   subsystem   that   we   denote  by $\overline\Sigma_\sigma$ (resp.\ 
$\overline\Sigma_s$).

We need now the following generalization of \cite[chap VI, \S1.1, lemme1]{bbk}
\begin{lemma} \label{bbk}
Let $\cL$ be a finite set of lines generating a vector space $V$ over a field
of characteristic 0; then two reflections of $V$
which stabilize $\cL$ and have a common eigenvalue $\zeta\neq 1$ with
$\zeta$-eigenspace the same line of $\cL$ are equal.
\end{lemma}
\begin{proof}
Here we mean by reflection an element $s\in\GL(V)$ such that $\ker(s-1)$ is a
hyperplane.
Let $s$ and $s'$ be reflections as in the statement. The product $s\inv s'$
stabilizes $\cL$, so has a power which fixes $\cL$,
thus is semi-simple. On the other hand $s\inv s'$ by assumption fixes one line 
$L\in \cL$ and induces the identity on $V/L$,
thus is unipotent. Being semi-simple and unipotent it has to be the
identity.
\end{proof}

It follows from \ref{bbk} that two root systems with proportional roots have
same Weyl group, thus same good primes; thus:
\begin{itemize}
\item $\Sigma_s$ and $\overline\Sigma_s$ have same good primes, as well
as $\Sigma_\sigma$ and $\overline\Sigma_\sigma$.
\item The bad primes for $\overline\Sigma_s$ are a subset of those for
$\overline\Sigma$, since it is a closed subsystem.
\end{itemize}
It only remains to show that the good primes for $\overline\Sigma$
are the same as for $\overline\Sigma_\sigma$, which can be checked case by case:
we can
reduce to the case where $\Sigma$ is irreducible, where these systems
coincide excepted when $\Sigma$ is of type $A_{2n}$; but in this case
$\overline\Sigma$ is of type $B_n$ and $\Sigma_\sigma$ is of type $B_n$ or 
$C_n$, which have the same set $\{ 2\}$ of bad primes.
\end{proof}

Since the characteristic is good for $\Gso$, hence also for
$(\bG^{\sigma l})^0$ by lemma \ref{good}, 
the elements $v'$ and $v$ are conjugate in $\GlsosF$ 
(see \cite[IV Corollaire 1.2]{memoire}).
By Lemma \ref{Sh vlsigma}(ii),
the element $v^\sFti$ runs over the unipotent elements
of $\GlpsosF$ when $v$ runs over $\GlsosF_\unip$.
Using moreover the equality
$|\GlsosF|=|\GlpsosF|$  we get 
\begin{multline}\tag{*}
(\sRLG\ShFsF(\chi))(\sigma lu)=
\frac 1{|\GlpsosF|}\sum_{u_1\in\GlpsosF_\unip}\chi(u_1l'\sigma)\\
\Trace((u_1,(u^\sFti)\inv)|H^*_c(\YUF)).\end{multline}
On the other hand by Lemma \ref{v'} applied with $v=u$, we have
$$(\ShFsF\sRLG(\chi))(\sll u)=
\sRLG(\chi)((\sll u)^\sFti)=\sRLG(\chi)(\lps\cdot
u^\sFti),$$
the second equality by Lemma \ref{Sh vlsigma}(i).
By the character formula this is equal to the right-hand side of 
formula~(*).
\end{proof}


\begin{thebibliography}{DM85}
\bibitem[Bou]{bbk} N.~Bourbaki, Groupes et alg\`ebres de Lie,
Chap. 4,5 et 6, {\sl Masson} (1981).

\bibitem[D99]{digne} F.~Digne,
Descente de Shintani et restriction des scalaires,
{\sl J. LMS \bf 59} (1999), 867--880.

\bibitem[DM85]{memoire}F.~Digne and J.~Michel,
Fonctions L des vari\'et\'es de Deligne-Lusztig
et descente de Shintani, {\sl M\'emoires SMF \bf 20} (1985), 1--144

\bibitem[DM91]{DM2} 
F.~Digne and J.~Michel, Representations of finite groups of Lie type,
{\sl LMS Student texts \bf 21} (1991).

\bibitem[DM94]{grnc}
F.~Digne and J.~Michel, Groupes r\'eductifs non connexes,
{\sl Annales de l'\'Ecole Normale Sup\'erieure \bf 20} (1994),
345--406.

\bibitem[DM02]{qss}
F.~Digne and J.~Michel, Points fixes des automorphismes \qss s,
{\sl C.R.A.S. \bf 334} (2001), 1055--1060.

\bibitem[E96]{eftekhari}
M.~Eftekhari, Faisceaux caract\`eres sur les groupes non connexes,
{\sl J. Algebra \bf 184} (1996), 516--537.

\bibitem[LLS]{LLS}
R.~Lawther, M.~Liebeck and G.~Seitz, ``Outer unipotent classes in 
automorphism groups of simple algebraic groups'', {\sl Proceedings of the
L.M.S. \bf{109}} (2014), 553--595.

\bibitem[Le]{Le} 
G.~Lehrer, Rational tori, semi-simple orbits and the topology of
hyperplane complements, {\sl Comment. Math. Helvetici \bf 67} (1992),
226--251.

\bibitem[Lu]{Lu} 
G.~Lusztig,   Character   sheaves   on   disconnected   groups  I--X,  {\sl
Representation  Theory \bf 7}  (2003) 374--403, {\bf  8} (2004) 72--178 and
346--413,  {\bf 9} (2005)  209--266, {\bf 10}  (2006) 314--379 and {\bf 13}
(2009) 82--140.

\bibitem[St]{St} 
R.~Steinberg,
Endomorphisms of linear algebraic groups,
{\sl Memoirs of the AMS \bf 80} (1968).
\end{thebibliography}
\end{document}